\documentclass{article}
\pdfoutput=1
\usepackage[utf8]{inputenc}
\usepackage{mathtools}
\usepackage{amsmath}
\usepackage{amsthm}
\usepackage{amssymb}
\usepackage{booktabs}
\usepackage{adjustbox}
\usepackage{ltablex,array}
\usepackage{longtable}
\usepackage{makecell}
\usepackage{marvosym}
\usepackage{hyperref}
\newcommand*{\faktor}[2]{
  \raisebox{0.5\height}{\ensuremath{#1}}
  \mkern-5mu\diagup\mkern-4mu
  \raisebox{-0.5\height}{\ensuremath{#2}}
} 
\newcommand\restr[2]{{
\left.
\kern-
\nulldelimiterspace 
#1 
\right|_{#2} 
}}
\usepackage{tikz}
\usepackage{tikz-cd}
\tikzcdset{scale cd/.style={every label/.append style={scale=#1},
    cells={nodes={scale=#1}}}}
\usepackage{quiver}
\usetikzlibrary{positioning} \usetikzlibrary{decorations.pathreplacing}
\usetikzlibrary{arrows}
\tikzset{>=stealth}
\newtheorem{manualtheoreminner}{Theorem}
\newenvironment{manualtheorem}[1]{%
  \renewcommand\themanualtheoreminner{#1}%
  \manualtheoreminner
}{\endmanualtheoreminner}
\newtheorem{theorem}{Theorem}[section]
\newtheorem{proposition}[theorem]{Proposition}
\newtheorem{lemma}[theorem]{Lemma}
\newtheorem{corollary}[theorem]{Corollary}

\theoremstyle{definition}
\newtheorem{remark}[theorem]{Remark}
\newtheorem{definition}[theorem]{Definition}

\theoremstyle{definition}
\usepackage[
backend=biber,
style=alphabetic,
sorting=ynt
]{biblatex}
\usepackage{url}
\begin{document}
\title{The (2,1)-category of small coherent categories}
\author{Kristóf Kanalas}
\maketitle

\begin{abstract}
    There is a well-known correspondence between coherent theories (and their interpretations) and coherent categories (resp.~functors), hence the (2,1)-category $\mathbf{Coh_{\sim}}$ (of small coherent categories, coherent functors and all natural isomorphisms) is of logical interest. We prove that this category admits all small 2-limits and 2-colimits (in the ($\infty $,1)-categorical sense), and prove a 2-categorical small object argument to provide weak factorisation systems for coherent functors.
\end{abstract}

\tableofcontents

\section{Introduction}

The class of coherent theories is one that has fairly nice model-theoretic properties (investigated under the name of positive logic), as well as it admits a useful categorical description, hence the toolkit of category theory can also be applied. Moreover this class is suitably general, as any first-order theory can be replaced by a coherent one (with the same class of $\mathbf{Set}$-models). This process is known as Morleyization in the literature.

The categorical viewpoint has a long history, the main developments are the syntactic approach (identifying theories and categories, see e.g.~\cite{makkai}), the model theoretic one (characterising the categories of models in terms of accessibility, as it is given in \cite{rosicky}), the geometric one (lifting the syntactic approach to the level of topos theory, discussed in \cite{sheaves}) and the computational one (type theory). Our work belongs to the first topic, some of whose results are summarised in the first section. 

We will denote the 2-category of coherent categories, coherent functors and all natural transformations by $\mathbf{Coh}$, while $\mathbf{Coh_{\sim }}$ will stand for the (2,1)-category, whose 2-cells are the natural isomorphisms. Our goal is to understand the 2-categorical structure of the latter, in particular to prove

\begin{manualtheorem}{4.12}
$\mathbf{Coh_{\sim}}$ is 2-complete and 2-cocomplete.
\end{manualtheorem}

This result, together with some properties of the underlying 1-category will imply a 2-categorical version of the small object argument for $\mathbf{Coh_{\sim}}$, which is hoped to form the first step towards a 2-dimensional model structure on the category of small coherent categories.

Finally we make the general remark, that most of our results apply in many other contexts, we restrict our attention to $\mathbf{Coh}$ for simplicity. First of all, everything could be stated for a lower level of logical complexity, which would result the same theorems for lex, regular, etc.~categories, with essentially the same (or simpler) proofs. Secondly, the requirement of strictness in the discussion of $(2,1)$-categories could be omitted, at the price that one has to triangulate each diagrams and resist using the terms "(homotopy) commutative square", "pentagon", etc.

\section{Some notions of categorical logic}

In this section we summarise the connection between coherent (also called positive) logic and coherent categories. Everything is taken from \cite{makkai}.

\begin{definition}
A formula is \emph{coherent} if it is built up from atomic formulas using $\wedge , \vee $ and $\exists $.
\end{definition}

\begin{definition}
A formula of the form $\forall x_1 \dots \forall x_n (\varphi \to \psi ) $ is also written as $\varphi \Rightarrow \psi $ and it is called a \emph{sequent}. If $\varphi $ and $\psi $ are coherent formulas then the sequent is said to be coherent.
\end{definition}

\begin{definition}
A \emph{coherent theory} is a set of coherent sequents. 
\end{definition}

\begin{definition}
Given a signature $L=\langle S,\mathbf{R},\mathbf{F} \rangle$, where $S$ is the set of sorts, $\mathbf{R}$ is the set of relation symbols (i.e.~has elements of the form $R\subseteq s_1\times \dots \times s_n$ with $s_i$'s being the related sorts), and $\mathbf{F}$ is the set of function symbols (e.g.~$f:s_1\times \dots \times s_n \to s$), an \emph{$L$-structure} in a category $\mathcal{C}$ associates to each sort $s$ an object $M(s)$ of $\mathcal{C}$, to each relation symbol $R$ a subobject $M(R) \leq M(s_1)\times \dots \times M(s_n)$, and to a function symbol $f$ a morphism $M(f):M(s_1)\times \dots \times M(s_n)\to M(s)$. 
\end{definition}

Our next goal is to interpret first-order formulas in structures (inside a fixed category $\mathcal{C}$). Depending on the complexity of our formula, this will require some extra structure on $\mathcal{C}$.

\begin{definition}
Let $M$ be an $L$-structure in a category $\mathcal{C}$. The \emph{interpretation} of a formula is given by the following steps:
\begin{itemize}
\item If $\vec{x}=(x_1,\dots x_n)$ is a finite sequence of free variables, $x_i$ is of sort $s_i$, then $M(\vec{x})=M(s_1)\times \dots \times M(s_n)$.
\item If $t$ is a term (of sort $s$) whose free variables are from $\vec{x}$, then $M_{\vec{x}}(t)$ will be an arrow $M(\vec{x})\to M(s)$ in the following way:
\begin{itemize}
\item If $t=x_i$, then $M_{\vec{x}}(t)$ is the projection map $M(\vec{x})\to M(x_i)=M(s)$.
\item If $t=f(t_1,\dots t_n)$, then $M_{\vec{x}}(t)$ is the composite $M(\vec{x})\xrightarrow {\langle M_{\vec{x}}(t_1),\dots \rangle} \prod M(s_i) \xrightarrow{M(f)} M(s)$
\end{itemize}
When $\mathcal{C}=\mathbf{Set}$, these are the functions which for a possible evaluation of $\vec{x}$ assign the induced value of $t$.
\item If $\varphi $ is a formula, whose free variables are along $\vec{x}=(x_1,\dots x_n)$, then its interpretation in the context $\vec{x}$ (if it exists) is a subobject $M_{\vec{x}}(\varphi )\leq M(\vec{x})$. It should be readily checked that in the case of $\mathbf{Set}$-models this gives precisely the set of evaluations of $\vec{x}$ which make $\varphi $ valid in $M$.
\begin{itemize}
\item \begin{tikzpicture}
\node (1) at (0,0) {$M_{\vec{x}}(t_1\approx t_2) $};
\node[anchor=base,right=10mm of 1] (2) {$M(\vec{x})$};
\node[anchor=base,right=10mm of 2] (3) {$M(s)$};
\draw[right hook ->] (1)--(2) node [midway,above] {\small{$e$}};
\draw[postaction={transform canvas={yshift=-1mm},draw}] [->] (2) -- (3) node [midway,above] {\small{$M_{\vec{x}}(t_1)$}} node [midway,below] {\small{$M_{\vec{x}}(t_2)$}};
\end{tikzpicture}
is an equalizer.
\item 
\begin{tikzpicture}
\node (3) at (0,0) {$M_{\vec{x}}(R(t_1,\dots t_n))$};
\node[anchor=base,right=10mm of 3] (2a) {$M(\vec{x})$};
\node[anchor=base,below=10mm of 3] (2b) {$M(R)$};
\node[anchor=base,below=10mm of 2a] (1) {$\prod _{i=1}^n M(s_i)$};
\draw[right hook->] (3)--(2a);
\draw[->] (3)--(2b);
\draw[->] (2a)--(1) node [midway,right] {$\langle M_{\vec{x}}(t_1),\dots \rangle$};
\draw[right hook->] (2b)--(1) node [midway,below] {$M(i)$};
\end{tikzpicture}
is a pullback (with $M(i)$ being the subobject $M(R)\hookrightarrow \prod _{i=1}^n M(S_i)$).
\item $M_{\vec{x}}(\bigwedge \Theta)=\bigwedge\{M_{\vec{x}}(\theta):\theta\in \Theta\}$ is the infimum (pullback) of the subobjects $M_{\vec{x}}(\theta)$.
\item $M_{\vec{x}}(\bigvee \Theta)=\bigvee\{M_{\vec{x}}(\theta):\theta\in \Theta\}$ is the supremum of the subobjects $M_{\vec{x}}(\theta)$.
\item $M_{\vec{x}}(\exists y \varphi)$ (where $y$ is not in $\vec{x}$) is the surjective-mono factorisation:
\begin{tikzpicture}
\node (3) at (0,0) {$M_{\vec{x},y}(\varphi)$};
\node[anchor=base,right=10mm of 3] (2a) {$M(\vec{x},y)$};
\node[anchor=base,right=10mm of 2a] (2b) {$M(\vec{x})$};
\node[anchor=base,below=10mm of 2a] (1) {$M_{\vec{x}}(\exists y \varphi)$};
\draw[right hook->] (3)--(2a);
\draw[->] (2a)--(2b) node [midway,above] {$\pi _{\vec{x}}$};
\draw[->>] (3)--(1);
\draw[right hook->] (1)--(2b);
\end{tikzpicture}

(An arrow $f:A\to B$ of $\mathcal{C}$ is surjective iff whenever it factors through a subobject $i:B'\hookrightarrow B$, we get that $i$ is an isomorphism.)
\item $M_{\vec{x}}(\neg \varphi)$ is the biggest (i.e.~contains every other such) subobject $A$ of $M(\vec{x})$, such that $A\wedge M_{\vec{x}}(\varphi) \leq 0_{M(\vec{x})}$, where $0_{M(\vec{x})}$ is the smallest subobject of $M(\vec{x})$.
\item $M_{\vec{x}}(\varphi\to \psi )$ is the biggest subobject $A$ of $M(\vec{x})$ such that $A\wedge M_{\vec{x}}(\varphi) \leq M_{\vec{x}}(\psi )$.
\item $M_{\vec{x}}(\forall y \varphi)$ is the biggest subobject $A$ of $M(\vec{x})$ such that $\pi _{\vec{x}}^{-1}(A)\leq M_{\vec{x},y}(\varphi )$. $\pi _{\vec{x}}^{-1}(A)$ denotes the pullback of $A$ along the projection $\pi _{\vec{x}}:M(\vec{x},y)\to M(\vec{x})$
\end{itemize}
\end{itemize}
\label{intpret}
\end{definition}

\begin{definition}
The sequent $\varphi \Rightarrow \psi$ is \emph{valid} in the structure $M$ (in symbols: $M\models \varphi \Rightarrow \psi $), iff $ M_{\vec{x}}(\varphi ) \leq M_{\vec{x}}(\psi ) $ (where $\vec{x}$ is the collection of all free variables in $\varphi \Rightarrow \psi $).

$M$ is a \emph{model} of the theory $T$ iff all the sequents from $T$ (have interpretation and) are valid in $M$.

A \emph{homomorphism} $\alpha :M\to M'$ of $T$-models consists of an arrow $\alpha _s: M(s)\to M'(s)$ for each sort $s$, for which the square

\[\begin{tikzcd}
	{M(s_1)\times \dots \times M(s_n)} && {M(s)} \\
	\\
	{M'(s_1)\times \dots \times M'(s_n)} && {M'(s)}
	\arrow["{M(f)}", from=1-1, to=1-3]
	\arrow["{\alpha _{s_1} \times \dots \times \alpha _{s_n}}"', from=1-1, to=3-1]
	\arrow["{\alpha _s}", from=1-3, to=3-3]
	\arrow["{M'(f)}", from=3-1, to=3-3]
\end{tikzcd}\]
commutes and the dashed arrow in

\[\begin{tikzcd}
	{M(R)} && {M(s_1)\times \dots \times M(s_n)} \\
	{\alpha (M(R))} \\
	{M'(R)} && {M'(s_1)\times \dots \times M'(s_n)}
	\arrow[hook, from=1-1, to=1-3]
	\arrow["{\alpha _{s_1} \times \dots \times \alpha _{s_n}}", from=1-3, to=3-3]
	\arrow[hook, from=3-1, to=3-3]
	\arrow[two heads, from=1-1, to=2-1]
	\arrow[hook, from=2-1, to=3-3]
	\arrow[dashed, from=2-1, to=3-1]
\end{tikzcd}\]
exists.

The category of $T$-models and homomorphisms in a category $\mathcal{C}$ is denoted by $T$-$mod(\mathcal{C})$.
\label{modelincat}
\end{definition}

\begin{remark}
For the interpretation of coherent logic, it is enough to assume that $\mathcal{C}$ has finite limits, finite sups and surjective-mono factorisation.
\end{remark}

This observation motivates the following definitions:

\begin{definition}
An arrow $f:X\to Y$ in a category $\mathcal{C}$ is an \emph{effective epimorphism} if the pullback
\begin{center}

\begin{tikzpicture}
\node (3) at (0,0) {$X\times _{Y} X$};
\node[anchor=base,right=10mm of 3] (2a) {$X$};
\node[anchor=base,below=10mm of 3] (2b) {$X$};
\node[anchor=base,below=10mm of 2a] (1) {$Y$};
\draw[->] (3)--(2a) node [midway,above] {$\pi$};
\draw[->] (3)--(2b) node [midway,left] {$\pi '$};
\draw[->] (2a)--(1) node [midway,right] {$f$};
\draw[->] (2b)--(1) node [midway,below] {$f$};
\end{tikzpicture}

\end{center}
exists and
\begin{center}

\begin{tikzpicture}
\node (1) at (0,0) {$X\times _{Y} X$};
\node[anchor=base,right=10mm of 1] (2) {$X$};
\node[anchor=base,right=10mm of 2] (3) {$Y$};
\draw[->] (2)--(3) node [midway,above] {\small{$f$}};
\draw[postaction={transform canvas={yshift=-1mm},draw}] [->] (1) -- (2) node [midway,above] {\small{$\pi$}} node [midway,below] {\small{$\pi '$}};
\end{tikzpicture}

\end{center}
is a coequalizer.

\label{effepi}
\end{definition}

\begin{definition}
A \emph{category} $\mathcal{C}$ is \emph{coherent}, if it
\begin{itemize}
\item has finite limits,
\item has images, i.e.~every morphism can be factored as an effective epimorphism followed by a monomorphism,
\item for any object $X$, the poset of its subobjects $Sub(X)$ is a lattice,
\item effective epimorphisms are pullback-stable,
\item for any map $f:X\to Y$, the induced map $f^{-1}:Sub(Y)\to Sub(X)$ is a lattice homomorphism.
\end{itemize}

A \emph{functor} $F:\mathcal{C}\to \mathcal{D}$ is \emph{coherent} if it
\begin{itemize}
\item preserves finite limits,
\item preserves effective epimorphisms,
\item the induced map $Sub(X)\to Sub(F(X))$ is a lattice homomorphism.
\end{itemize}

The 2-category of small coherent categories, coherent functors and all natural transformations is denoted by $\mathbf{Coh}$.
\end{definition}

\begin{remark}
As $F$ preserves finite limits, it follows that monomorphisms are also preserved (since $f:X\to Y$ is mono iff
\begin{center}
\begin{tikzpicture}
\node (3) at (0,0) {$X$};
\node[anchor=base,right=10mm of 3] (2a) {$X$};
\node[anchor=base,below=10mm of 3] (2b) {$X$};
\node[anchor=base,below=10mm of 2a] (1) {$Y$};
\draw[->] (3)--(2a) node [midway,above] {$1_X$};
\draw[->] (3)--(2b) node [midway,left] {$1_X$};
\draw[->] (2a)--(1) node [midway,right] {$f$};
\draw[->] (2b)--(1) node [midway,below] {$f$};
\end{tikzpicture}
\end{center}
is a pullback), so the map $Sub(X)\to Sub(F(X))$ makes sense.
\end{remark}

\begin{remark}
An arrow $f:X\to Y$ in a coherent category is surjective iff it is an effective epimorphism.
\end{remark}

The first part of the proposed correspondence is to replace categories with theories:

\begin{definition}
The \emph{canonical language} of the category $\mathcal{C}$ has the signature $L=L_{\mathcal{C}}$ which contains a sort $\bar{A}$ for every object $A$ of $\mathcal{C}$, and a function symbol $\bar{f}:\bar{A}\to \bar{B}$ for every such arrow $f$ of $\mathcal{C}$ (and nothing else). Then $\mathcal{C}$ is naturally an $L$-structure by the identical interpretation of $L$ (i.e.~sending $\bar{A}$ to $A$ and $\bar{f}$ to $f$). More generally; each functor $F:\mathcal{C}\to \mathcal{D}$ creates an $L$-structure in $\mathcal{D}$.
\end{definition}

The following theorem (2.4.5. in \cite{makkai}) says, that from inside, $\mathcal{C}$ looks similar to $\mathbf{Set}$.

\begin{theorem}
Assume, that $\mathcal{C}$ has finite limits. Then the following diagrams in $\mathcal{C}$ have the stated properties, iff the sequents on their right side (have interpretation and) are valid (in $\mathcal{C}$, as a structure over its canonical language, with the identical interpretation for the signature).

\begin{tabularx}{\textwidth}{cXX}
1. & $A\xrightarrow{f} A$ is the identity on $A$ & $\Rightarrow f(a)\approx a$ \\ 
\hline 
2. & 
\begin{tikzpicture}
\node (3) at (0,0) {$A$};
\node[anchor=base,right=5mm of 3] (2a) {$C$};
\node[anchor=base,below=5mm of 3] (2b) {$B$};
\draw[->] (3)--(2a) node [midway,above] {$h$};
\draw[->] (3)--(2b) node [midway,left] {$f$};
\draw[->] (2b)--(2a) node [midway,below] {$g$};
\end{tikzpicture} is commutative & $\Rightarrow gf(a)\approx h(a)$\\ 
\hline 
3. & $A\xrightarrow{f} B$ is mono & $f(a)\approx f(a') \Rightarrow a\approx a'$ \\ 
\hline 
4. & $A\xrightarrow{f} B$ is surjective & $\Rightarrow \exists a:f(a)\approx b$ \\ 
\hline 
5. & $A$ is the terminal object & $\Rightarrow a\approx a'$ \\ & & $\Rightarrow \exists a: a\approx a$ \\ 
\hline 
6. & $A$ is the initial object & $a\approx a \Rightarrow $ \\ 
\hline 
7. & $A\xleftarrow{f} C\xrightarrow{g} B$ is a product diagram & \small{$f(c)\approx f(c') \wedge g(c)\approx g(c') \Rightarrow c\approx c'$} \\ & & $\Rightarrow \exists c (f(c)\approx a \wedge g(c)\approx b) $ \\ 
\hline 
8. & \begin{tikzpicture}
\node (1) at (0,0) {$E$};
\node[anchor=base,right=5mm of 1] (2) {$A$};
\node[anchor=base,right=5mm of 2] (3) {$B$};
\draw[right hook->] (1)--(2) node [midway,above] {$\epsilon$};
\draw[postaction={transform canvas={yshift=-1mm},draw}] [->] (2) -- (3) node [midway,above] {$f$} node [midway,below] {$g$};
\end{tikzpicture} is an equalizer & $f(a)\approx g(a) \Leftrightarrow \exists e: \epsilon (e)\approx a$ \\ 
\hline 
9. & \makecell[l]{$B\xhookrightarrow{g} X\xhookleftarrow{f_i} A_i$ ($i\in I$). \\ $B$ is the sup of $A_i$-s} & \small{$\bigvee _{i\in I} \exists a_i: f_i(a_i)\approx x \Leftrightarrow \exists b: g(b)\approx x$}\\ 
\hline 
10. & \makecell[l]{$B\xhookrightarrow{g} X\xhookleftarrow{f_i} A_i$ ($i\in I$). \\ $B$ is the inf of $A_i$-s} & \small{$\bigwedge _{i\in I} \exists a_i: f_i(a_i)\approx x \Leftrightarrow \exists b: g(b)\approx x$} \\ 
\end{tabularx}
\label{diagram}
\end{theorem}

\begin{definition}
Let $\mathcal{C}$ be a coherent category. Its (coherent) \emph{internal theory} $T_{\mathcal{C}}$ (or $Th(\mathcal{C})$) over the signature $L_{\mathcal{C}}$ consists of those sequents which refer to identities, commutative triangles, finite limits, surjective arrows and finite unions (as it is described above).
\end{definition}

\begin{theorem}
The categories $T_{\mathcal{C}}$-$mod(\mathcal{E})$ and $\mathbf{Coh}(\mathcal{C},\mathcal{E})$ are isomorphic.
\end{theorem}

Now we replace theories with categories. The notion of derivability ($\vdash $) refers to a deduction system which is sound wrt.~every coherent category and which is complete wrt.~Boolean-valued $\mathbf{Set}$-models, see \cite{makkai} for the details.

\begin{definition}
Let $T$ be a coherent theory. Its \emph{syntactic category} $\mathcal{C}_T$ is defined as follows:
\begin{itemize}
\item The objects are equivalence classes of coherent formulas (in context) over the given signature, where $\varphi (\vec{x})\sim \psi (\vec{y})$, iff $\psi (\vec{y})=\varphi (\vec{y}/ \vec{x})$. Note that $[\varphi (\vec{x})]$ and $[\varphi (\vec{x},\vec{y})]$ (with $y$ being an extra variable not present in $\varphi$) corresponds to different objects. This technicality is not essential, as $[\varphi (\vec{x},y)]$ turns out to be isomorphic with $[\varphi (\vec{x}) \wedge y\approx y]$, hence if we require all variables in the context $\vec{x}$ to appear freely in $\varphi $ we get an equivalent category.
\item An arrow $[\varphi (\vec{x})]\xrightarrow{[\theta (\vec{x},\vec{y})]} [\psi (\vec{y})]$ is an equivalence class of formulas, having the following properties:
\begin{itemize}
\item $\vec{x}$ and $\vec{y}$ are disjoint (this can always be assumed, as we can find such representatives of the objects),
\item $T\vdash \theta(\vec{x},\vec{y}) \Rightarrow \varphi (\vec{x}) \wedge \psi (\vec{y})$,
\item $T\vdash \varphi (\vec{x}) \Rightarrow \exists \vec{y} \theta(\vec{x},\vec{y})$,
\item $T\vdash \theta(\vec{x},\vec{y}) \wedge \theta(\vec{x},\vec{y'}) \Rightarrow \vec{y}=\vec{y'}$.
\end{itemize}
$\theta(\vec{x},\vec{y}) \sim \theta'(\vec{x'},\vec{y'})$, iff $T\vdash \theta(\vec{x},\vec{y}) \Leftrightarrow \theta'(\vec{x},\vec{y})$.
\end{itemize}
\end{definition}

\begin{remark}
The required properties for $\theta $ are often referred as being "T-provably functional". This is because these are exactly the conditions which can guarantee, that the interpretation of $\theta $ in a model $M$ is not merely a subobject of $M(\varphi )\times M(\psi )\leq M(\vec{x})\times M(\vec{y})$, but the graph of an arrow from $M(\varphi )$ to $M(\psi )$.
\end{remark}

\begin{theorem}
Given a coherent theory $T$, its syntactic category $\mathcal{C}_T$ is a well-defined coherent category. The categories $T$-$mod(\mathcal{E})$ and $\mathbf{Coh}(\mathcal{C}_T,\mathcal{E})$ are equivalent. If $\mathcal{C}$ is coherent then $\mathcal{C}_{Th(\mathcal{C})}$ and $\mathcal{C}$ are equivalent.
\end{theorem}

\section{Limits and weak colimits}

The main ingredient for the existence of weak colimits is the following fact:

\begin{proposition}
Given an arbitrary subcategory $i:\mathcal{C}\hookrightarrow \mathcal{E}$ of a coherent category $\mathcal{E}$, it is included in a coherent subcategory $\tilde{\mathcal{C}}\hookrightarrow \mathcal{E}$ (with coherent embedding), such that $|\tilde{\mathcal{C}}|\leq \aleph _0 \cdot |\mathcal{C}|$. (Let's say $|\mathcal{C}|$ is defined to be $|Arr(\mathcal{C})|$.)
\label{solution}
\end{proposition}

\begin{proof}
Form the theory $T$ over the signature $L_{\mathcal{C}}$ which consists of the sequents for identities and commutative triangles. Then $i$ corresponds to a $T$-model in $\mathcal{E}$, which induces a coherent functor $\mathcal{C}_T\to \mathcal{E}$, whose image has a size smaller or equal to $\aleph _0 \cdot |L_{\mathcal{C}}|=\aleph _0\cdot |\mathcal{C}|$ and contains $\mathcal{C}$ as a subcategory.
\end{proof}

We start with the existence of limits.

\begin{proposition}
$\mathbf{Coh}$ has pullbacks along isofibrations, and the forgetful functor $U:\mathbf{Coh}\to \mathbf{Cat}$ preserves them.
\label{cohpb}
\end{proposition}

\begin{lemma}
Take a diagram of the form $\mathcal{C}\xrightarrow{F}\mathcal{D}\xleftarrow{F'}\mathcal{C'}$ in $\mathbf{Cat}$, where either $F$ or $F'$ is an isofibration. If all three categories have certain types of limits or colimits, and these are preserved by both functors, then $\mathcal{C}\times _{\mathcal{D}}\mathcal{C'}$ will also admit these (co)limits, and they will be preserved by the projection maps. Moreover they are reflected by the pair of the projections.
\label{pblim}
\end{lemma}

\begin{proof}
Take a diagram of that fixed type (i.e.~a functor from the fixed index category) in $\mathcal{C}\times _{\mathcal{D}}\mathcal{C'}$. It is sent to the same type of diagrams in $\mathcal{C}$ and $\mathcal{C'}$, so by assumption, we can take a (co)limiting (co)cone over them, which are preserved by $F$ and $F'$. As (co)limits are unique up to unique isomorphism, there is an isomorphism $(i)$ in $\mathcal{D}$, that connects the tips of the images and makes everything commute. As (let's say) $F$ was an isofibration, we can modify our (co)cone in $\mathcal{C}$ by the composition of a preimage of $i$, hence we can assume, that the images at $F$ and at $F'$ are the same. By the pullback-construction, there is a (co)cone in $\mathcal{C}\times _{\mathcal{D}}\mathcal{C'}$, that is mapped to the chosen (co)limiting ones by the projections. A similar argument shows that it must be a (co)limiting one, the preservation (and the reflection by the pair) follows from the construction.
\end{proof}

\begin{lemma}
The coequalizer of kernel pairs (as in Definition \ref{effepi}) exists in every coherent category, and it is preserved by every coherent functor.
\label{coeq}
\end{lemma}

\begin{proof}
Using the notation of Definition \ref{effepi}, factor $f$ as $X\xrightarrow{e} \bar{X} \xrightarrow{i} Y$, where $e$ is an effective epi and $i$ is mono. It follows that
\begin{center}
\begin{tikzpicture}
\node (3) at (0,0) {$X\times _Y X$};
\node[anchor=base,right=10mm of 3] (2a) {$X$};
\node[anchor=base,below=10mm of 3] (2b) {$X$};
\node[anchor=base,below=10mm of 2a] (1) {$\bar{X}$};
\draw[->] (3)--(2a) node [midway,above] {$\pi$};
\draw[->] (3)--(2b) node [midway,left] {$\pi '$};
\draw[->] (2a)--(1) node [midway,right] {$e$};
\draw[->] (2b)--(1) node [midway,below] {$e$};
\end{tikzpicture}
\end{center}
is also a pullback, hence $e=coeq(\pi ,\pi ')$, as by definition $e$ is the coequalizer of its kernel pair.

This also proves the uniqueness (up to unique isomorphism) of such factorisations. 
\end{proof}

\begin{proof}[Proof of Proposition \ref{cohpb}]
We will use the notation of Lemma \ref{pblim}, and check that $\mathcal{C}\times _{\mathcal{D}} \mathcal{C'}$ is coherent:
\begin{itemize}
\item finite limits: Lemma \ref{pblim}
\item images: Using that $F$ was an isofibration, we get a factorisation in \\ $\mathcal{C}\times _{\mathcal{D}}\mathcal{C'}$ from the factorisations in $\mathcal{C}$ and in $\mathcal{C'}$. By the previous lemmas, the projections $\pi $ and $\pi '$ preserve, and together reflect effective epimorphisms and monomorphisms.
\item joins: By the existence of finite limits, it is enough to see that finite (possibly empty) joins of subobjects exist. This follows from a similar argument to the one in the proof of Lemma \ref{pblim}.
\item pullback-stability: $\pi $ and $\pi '$ preserve, and together reflect pullbacks.
\end{itemize}
\end{proof}

\begin{corollary}
$\mathbf{Coh}$ has finite products.
\label{finprod}
\end{corollary}

\begin{proof}
The unique map to the terminal object is an isofibration.
\end{proof}

\begin{remark}
Since limits, subobjects, unions, composition of arrows, etc. are defined coordinate-wise, it is equally easy to see the existence of arbitrary products.
\end{remark}

\begin{theorem}
$\mathbf{Coh}$ has weak colimits.
\label{complete}
\end{theorem}

\begin{proof}
Fix a diagram $d_{\bullet }:\mathcal{I}\to \mathbf{Coh}$, and let $F:\mathbf{Coh}\to \mathbf{Set}$ be the functor $lim _{\mathcal{I}^{op}} \mathbf{Coh}(d_i, -)$. The existence of a weak colimit is equivalent to the existence of a weak initial object in the category $*\downarrow F$ (see: \cite{maclane} Theorem V.6.3.). As $\mathbf{Coh}$ has products and $F$ preserves them, $*\downarrow F$ has products, so it is enough to find a weakly initial family, as in this case the product of its elements is a weak initial object.

An element of $lim _{\mathcal{I}^{op}} \mathbf{Coh}(d_i, C)$ is a compatible family of functors $\{f_i:d_i\to C\}$. To prove the existence of a weakly initial family, set $C'$ to be the subcategory with objects $\bigcup Ob(Im\ (f_i))$ and with arrows the composites of arrows from $\bigcup Arr(Im\ (f_i))$. The size of $C'$ is bounded by the size of $\sum |d_i|$, but it is not necessarily coherent.

Using Corollary \ref{solution}, the inclusion $C'\to C$ factors through a coherent subcategory $g: \tilde{C}\to C$, where the size of $\tilde{C}$ is still limited by $\sum |d_i|$. Now it follows that each $f_i$ factors as $g\circ \tilde{f_i}$, and in this case $\tilde{f_i}:d_i\to \tilde{C}$ must also be coherent. 

This shows that for some fixed $\kappa \geq \sum |d_i|$ the set of coherent categories of cardinality $\leq \kappa$ (one from each isomorphism class) is a solution set, namely the set of all cocones (over the fixed diagram) with top element having cardinality $\leq \kappa$ is a weakly initial family.
\end{proof}

\section{Filtered colimits}

First recall the following basic result on the construction of general colimits (cf.~\cite{maclane} Theorem V.2.1.).

\begin{theorem}
Let $d_\bullet :\mathcal{I}\to \mathcal{C}$ be a diagram, where $\mathcal{C}$ is a cocomplete category. Its colimit can be computed as the coequalizer

\begin{center}
\begin{tikzpicture}
\node (1) at (0,0) {$\bigsqcup _{f\in \mathcal{I}_1} dom(f)$};
\node[anchor=base,right=15mm of 1] (2) {$\bigsqcup _{A\in \mathcal{I}_0} A$};
\node[anchor=base,right=10mm of 2] (3) {$colim\ d_\bullet $};
\draw[->] (2)--(3) node [midway,above] {\small{$r$}};
\draw[postaction={transform canvas={yshift=-1mm},draw}] [->] (1) -- (2) node [midway,above] { \small{$\sqcup f$}} node [midway,below] { \small{$\sqcup 1_{dom(f)}$}};
\end{tikzpicture}
\end{center}
where the edges of the cocone are given by $A\xrightarrow{j_A} \bigsqcup _{A\in \mathcal{I}_0} A \xrightarrow{r} colim\ d_\bullet $.
\end{theorem}

As in the category $\mathbf{Set}$ we have a good understanding of coproducts (disjoint union) and coequalizers (factorisation by the equivalence relation generated by the pairs $(f(x),g(x))$), we can describe filtered colimits explicitly:

\begin{theorem}
Let $d_\bullet :\mathcal{I}\to \mathbf{Set}$ be a diagram, where $\mathcal{I}$ is a filtered category. Its colimit is the set $\faktor{\{(x,i):x\in d_i \}}{\sim}$, where $(x,i)\sim (y,j)$ iff $\exists \varphi _{ik}:i\to k $ $ \exists \varphi _{jk}:j\to k \ d(f)(x)=d(g)(y)$. The $i$-th coprojection is given by $d_i\ni x\mapsto [(x,i)]$.
\end{theorem}

The proof consists of the simple observation that the axioms of filtered categories are precisely the ones that force $\sim $ to be an equivalence relation. (In what follows we will not distinguish between objects and arrows in $\mathcal{I}$ and their image at $d_\bullet $.) Now we claim that the situation for $\mathbf{Cat}$ is essentially the same, in other words:

\begin{proposition}
The forgetful functor $V:\mathbf{Cat}\to \mathbf{Set}$ (which assigns to every category the set of its arrows) preserves filtered colimits. The identity arrow on an object $A$ of $colim\ d_\bullet$ is represented by $(1_A,i)$, where $A$ is in the image of the $i$-th coprojection. Similarly, the composition of $(f,i):(A,i)\to (B,i)$ and $(g,j):(B,i)=(C,j)\to (D,j)$ is $(\varphi _{jk} (g)\circ \varphi _{ik}(f), k)$, where $\varphi _{i/j,k}$ is a map $\mathcal{C}_{i/j}\to \mathcal{C}_k$ in the diagram, such that $\varphi _{ik}(B) =\varphi _{jk}(C)$.
\label{catset}
\end{proposition}

\begin{proof}
The above construction of identities and composition is well-defined, since each function in the diagram is a functor (hence we can choose any suitable $k$). Each coprojection map is also a functor. Given a cocone $(\mu _i)_{i\in \mathcal{I}_0}$ with all edges being functors, the induced (unique) function must also be a functor: any composable pair exists at some stage $\mathcal{C}_i$, hence it is mapped to the composition by $\mu _i$, and also by the induced universal map by commutativity. The same holds for identities.
\end{proof}

The final step is to prove

\begin{theorem}
$\mathbf{Coh}$ has filtered colimits and the forgetful functor $U:\mathbf{Coh}\to \mathbf{Cat}$ preserves them.
\label{cohcat}
\end{theorem}

\begin{proof}
First we have to prove that if all $\mathcal{C}_i$'s and all $\varphi _{ij}$'s are coherent then so is the colimit, with coherent coprojections.
\begin{itemize}
    \item finite limits: Given a finite diagram $\Delta \xrightarrow{\delta _\bullet} colim\ d_\bullet $, it factors through some $\mathcal{C}_i$. One can take the limit cone here, whose image at $\mu _i$ will be the limit of $\delta _\bullet $. I.e.~given another cone with top object $[(y,j)]$ (assuming the whole diagram with the two cones factors through $\mathcal{C}_j$), we can find an index $k$ and maps $\varphi _{ik}$, $\varphi _{jk}$ such that the image of $\delta _\bullet $ at $\mathcal{C}_i$ and at $\mathcal{C}_j$ is identified by them. Since $\varphi _{ik}$ preserves finite limits, we have an induced map $!:(x,k)\to (y,k)$, whose image at $\mu _k$ makes everything commute. Uniqueness is proved similarly. It follows that $\mu _i$'s preserve finite limits.
    \item image factorisation: By the previous argument monomorphisms are preserved by the coprojection maps and the image of the kernel pair of $f$ at $\mu _i$ is the kernel pair of $\mu _i(f)$. Since kernel pairs have coequalizers in all $\mathcal{C}_i$'s, and they are preserved by $\varphi _{ij}$'s, the above argument shows that they are also preserved by the coprojections, hence effective epimorphisms are preserved. Given an arrow in the colimit one can take any of its preimages, and by the previous argument its factorisation will be mapped to an image factorisation in $colim\ d_\bullet$.
    \item unions: Take an object $[(x,i)]$ in $colim\ d_\bullet$ and two monomorphism $[(m_1,i)]$, $[(m_2,i)]$ into it (again: we can assume that these $i$'s are the same). In $\mathcal{C}_i$ we see the same diagram, but we can not guarantee that $m_1$ and $m_2$ are monos. Instead we can factor them into an effective epi followed by a mono, and by the uniqueness of these factorisations (up to unique isomorphism), we see that the monic part will go to the same subobject which was represented by $[(m_1,i)]$ (resp.~$[(m_2,i)]$). Now we can take the union in $\mathcal{C}_i$, and check that its image at $\mu _i$ is the union in the colimit.
    \item pullback-stability: Every effective epimorphism (resp.~pullback square) comes from an effective epi (resp.~pullback) in some $\mathcal{C}_i$, hence we are done. The same works for unions.
\end{itemize}
\end{proof}

Our next goal is to prove that for every coherent category $\mathcal{C}$, the hom-functor $\mathbf{Coh}(\mathcal{C},-)$ commutes with $|\mathcal{C}|^+$-filtered colimits. This is known for the category $\mathbf{Set}$:

\begin{proposition}
For any set $A$ the functor $\mathbf{Set}(A,-)$ preserves $|A|^+$-filtered colimits.
\end{proposition}

\begin{proof}
The induced map $!: colim\ \mathbf{Set}(A,d_i) \to \mathbf{Set}(A,colim\ d_\bullet)$ has as domain the set of equivalence classes $[(f:A\to d_i)]$, where $f:A\to d_i$ is equivalent to $g:A\to d_j$ if there are maps $\varphi _{ik}$, $\varphi _{jk}$ in the image of $d_\bullet$, such that $\varphi _{ik}f=\varphi _{jk}g$. $!$ takes $[(f:A\to d_i)]$ to $A\xrightarrow{f} d_i \xrightarrow{\mu _i} colim\ d_\bullet$, and by the construction of an $|A|^+$-filtered colimit, it is automatically injective (i.e.~for each element $a$ there is a suitable $k$, and maps $\varphi _{ik}$, $\varphi _{jk}$, such that $\varphi _{ik}f(a)=\varphi _{jk}g(a)$, and these have a common upper-bound).

Given a function $f:A\to colim\ d_\bullet$, each $a\in A$ is included in the image of some $\mu _i$, by $|A|^+$-filteredness these $d_i$-s have a common extension $d_k$, hence $f$ factors through $A_k$. Therefore $!$ is also surjective.
\end{proof}

Now it follows easily for $\mathbf{Cat}$:

\begin{proposition}
For any category $\mathcal{C}$ the functor $\mathbf{Cat}(\mathcal{C},-)$ preserves $|\mathcal{C}|^+$-filtered colimits.
\end{proposition}

\begin{proof}
By Proposition \ref{catset} we can compute filtered colimits in $\mathbf{Set}$. If the map $!:colim\ \mathbf{Cat}(\mathcal{C},d_i) \to \mathbf{Cat}(\mathcal{C},colim\ d_\bullet)$ would take two elements to the same functor, then $!':colim\ \mathbf{Set}(V(\mathcal{C}),V(d_i)) \to \mathbf{Set}(V(\mathcal{C}),colim\ V\circ d_\bullet)$ would not be injective.

Given a functor $F:\mathcal{C}\to colim\ d_\bullet$, it factors through some $d_i$ as a(n arrow) function (e.g.~$F=\mathcal{C}\xrightarrow{F'}d_i \xrightarrow{\mu _i} colim\ d_\bullet$). Now we just count how many things can go wrong.

For each commutative triangle and each identity arrow it might happen that $F'$ does not preserve it, but in this case there is an index $k$ and an arrow $\varphi _{ik}$ in the diagram, such that $\varphi _{ik}F'$ corrects that mistake. Since there are at most $|\mathcal{C}|$-many commutative triangles and identities in $\mathcal{C}$ (or finitely many if $\mathcal{C}$ is finite), and our diagram is $|\mathcal{C}|^+$-filtered (or $\aleph _0$-filtered in the finite case), we are done.
\end{proof}

This shows how to proceed when the category $\mathbf{Coh}$ is considered:

\begin{proposition}
For any coherent category $\mathcal{C}$ the functor $\mathbf{Coh}(\mathcal{C},-)$ preserves $|\mathcal{C}|^+$-filtered colimits.
\label{filtered}
\end{proposition}

\begin{proof}
Again, $!$ is injective and given a coherent functor $F:\mathcal{C}\to colim\ d_\bullet$, it factors through some $d_i$ as $\mathcal{C}\xrightarrow{F'}d_i \xrightarrow{\mu _i} colim\ d_\bullet$, where $F'$ is a (not necessarily coherent) functor. That is, for each finite diagram, effective epimorphism, and pair of subobjects it might happen, that $F'$ does not preserve the limit, the effective epi and the union, but some $\varphi _{ij}$ corrigates one of these mistakes. Since there are at most $ |\mathcal{C}|$-many (or finitely many) such diagrams in $\mathcal{C}$, we can find a $\varphi_{ik}$ for which $\varphi _{ik}F'$ is coherent.
\end{proof}

\begin{theorem}
$\mathbf{Coh}$ is $\aleph _1$-accessible.
\end{theorem}

\begin{proof}
Proposition \ref{filtered} shows that every coherent category $\mathcal{C}$ is $|\mathcal{C}|$-presentable, and Proposition \ref{solution} implies that every coherent category is the union of its countable subcategories. Clearly the poset of countable subcategories is $\aleph _1$-filtered since $\aleph _0 \cdot \aleph _0 =\aleph _0$.
\end{proof}

\section{2-categorical aspects}

If an accessible category has all limits then it is cocomplete (see Corollary 2.47. of \cite{rosicky}). Hence if $\mathbf{Coh}$ had pullbacks we could derive that $\mathbf{Coh}$ is complete and cocomplete. This is not the case since e.g. the two inclusions $* \hookrightarrow (*\leftrightarrow *)$ have empty intersection and the map from the 2-element Boolean algebra $\mathbf{2}$ to a coherent category is unique only up to unique natural isomorphism. However these examples show that we can hope for completeness and cocompleteness in a 2-categorical sense. In this section we will show that this idea is right.

\begin{definition}
Given a small 2-diagram $d_\bullet :\mathcal{I}\to \mathcal{C}$ where $\mathcal{C}$ is an arbitrary (2,1)-category, its \emph{homotopy} (or 2-) \emph{limit} is a cone

$ $\\
\adjustbox{scale=0.8,center}
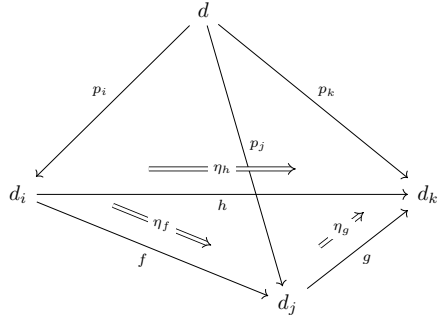
{
\begin{tikzcd}
	&&& d \\
	\\
	\\
	\\
	{d_i} &&&&&& {d_k} \\
	\\
	&&&& {d_j}
	\arrow["{f}"', from=5-1, to=7-5]
	\arrow["{g}"', from=7-5, to=5-7]
	\arrow["{h}"', from=5-1, to=5-7]
	\arrow["{p_i}"', from=1-4, to=5-1]
	\arrow["{p_j}", from=1-4, to=7-5]
	\arrow["{p_k}", from=1-4, to=5-7]
	\arrow["{\eta _f}"{description}, shift left=5, shorten <=34pt, shorten >=34pt, Rightarrow, from=5-1, to=7-5]
	\arrow["{\eta _g}"{description}, shift left=5, shorten <=17pt, shorten >=17pt, Rightarrow, from=7-5, to=5-7]
	\arrow["{\eta _h}"{description}, shift left=5, shorten <=53pt, shorten >=53pt, Rightarrow, from=5-1, to=5-7]
\end{tikzcd}
}
such that for each 2-cell $g\circ f \Rightarrow h$ in the diagram the above tetrahedron is filled by the identical 3-cell. Moreover it has the following universal property: given another such cone there is a map $r:e\to d$, unique up to unique natural isomorphism, together with 2-isomorphisms $\alpha _i : p_i r \Rightarrow q_i$

$ $\\
\adjustbox{scale=0.8,center}
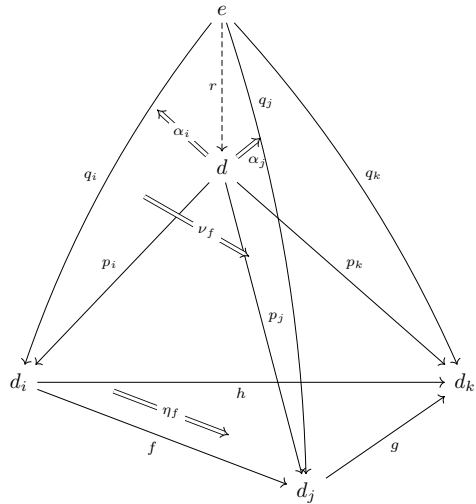
{
\begin{tikzcd}
	&&&&& {} \\
	&&& e \\
	\\
	&& {} && {} \\
	&& {} & d \\
	&&&&& {} \\
	&&&& {} &&&&& {} \\
	\\
	{d_i} &&&&&& {d_k} \\
	\\
	&&&& {d_j}
	\arrow["{f}"', from=9-1, to=11-5]
	\arrow["{g}"', from=11-5, to=9-7]
	\arrow["{h}"', from=9-1, to=9-7]
	\arrow["{p_i}"', from=5-4, to=9-1]
	\arrow["{p_j}", from=5-4, to=11-5]
	\arrow["{p_k}", from=5-4, to=9-7]
	\arrow["{\eta _{f}}"{description}, shift left=5, shorten <=34pt, shorten >=34pt, Rightarrow, from=9-1, to=11-5]
	\arrow["{q_i}"', curve={height=12pt}, from=2-4, to=9-1]
	\arrow["{q_k}", curve={height=-12pt}, from=2-4, to=9-7]
	\arrow[""{name=0, anchor=center, inner sep=0}, "{q_j}"{pos=0.2}, curve={height=-12pt}, from=2-4, to=11-5]
	\arrow["r"', dashed, from=2-4, to=5-4]
	\arrow["{\alpha _i}"{description}, Rightarrow, from=5-4, to=4-3]
	\arrow["{\alpha _j}"'{pos=0.1}, shorten >=22pt, Rightarrow, from=5-4, to=4-5]
	\arrow["{\nu _{f}}"{description}, shift right=5, shorten >=11pt, Rightarrow, from=5-3, to=0]
\end{tikzcd}
}
$ $\\
such that the composite of the 2-cells $\alpha _i: p_i r\Rightarrow q_i $, $\alpha _j : p_j r \Rightarrow q_j$ and $\eta _f: f p_i \Rightarrow p_j$ is $\nu _f : f q_i \Rightarrow q_j$ (for each arrow $f$ of the diagram).

\end{definition}

\begin{remark}
If $\mathcal{C}$ is a strict (2,1)-category then we can give a simpler description for certain homotopy limits. E.g.~a 2-pullback is a square 

\[\begin{tikzcd}
	{C^*} && C \\
	\\
	B && A
	\arrow["{g'}"', from=1-1, to=3-1]
	\arrow["f"', from=3-1, to=3-3]
	\arrow["{f'}", from=1-1, to=1-3]
	\arrow["g", from=1-3, to=3-3]
	\arrow["\eta"{description}, shorten <=11pt, shorten >=11pt, Rightarrow, from=3-1, to=1-3]
\end{tikzcd}\]
such that given an outer square with natural isomorphism $\nu :fh_1 \Rightarrow gh_2$ there is a map $r: D\to C^*$ (unique up to unique natural isomorphism) and isomorphisms $\alpha _1: g'r\Rightarrow h_1$, $\alpha _2: f'r\to h_2$
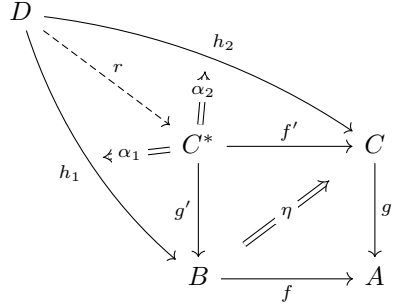
\[\begin{tikzcd}
	D \\
	\\
	&& {C^*} && C \\
	\\
	&& B && A
	\arrow["{g'}"', from=3-3, to=5-3]
	\arrow["f"', from=5-3, to=5-5]
	\arrow["{f'}", from=3-3, to=3-5]
	\arrow["g", from=3-5, to=5-5]
	\arrow["\eta"{description}, shorten <=11pt, shorten >=11pt, Rightarrow, from=5-3, to=3-5]
	\arrow[""{name=0, anchor=center, inner sep=0}, "{h_2}", curve={height=-12pt}, from=1-1, to=3-5]
	\arrow[""{name=1, anchor=center, inner sep=0}, "{h_1}"', curve={height=12pt}, from=1-1, to=5-3]
	\arrow["r", dashed, from=1-1, to=3-3]
	\arrow["{\alpha _1}"{description}, shorten >=5pt, Rightarrow, from=3-3, to=1]
	\arrow["{\alpha _2}"{description}, shorten >=5pt, Rightarrow, from=3-3, to=0]
\end{tikzcd}\]
such that the composition of the 2-cells $\alpha _1^{-1}$, $\alpha _2$ and $\eta $ is $\nu$.

\end{remark}

\begin{remark}
Our notion of a 2-(co)limit coincides with that of \cite{lurietopos} when $\mathcal{C}$ is regarded as an $(\infty ,1)$-category where the lifting in the diagram
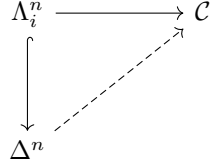
\[\begin{tikzcd}
	{\Lambda _i^n} && {\mathcal{C}} \\
	\\
	{\Delta ^n}
	\arrow[hook, from=1-1, to=3-1]
	\arrow[dashed, from=3-1, to=1-3]
	\arrow[from=1-1, to=1-3]
\end{tikzcd}\]
is unique for $n\geq 3$ and $0<i<n$.
\end{remark}

\begin{theorem}
$\mathbf{Coh_{\sim}}$ has all homotopy products and pullbacks.
\end{theorem}

\begin{proof}
First observe, that the (1-categorical) product is a homotopy limit. I.e.~as there are no arrows in the diagram, the notion of a(n ordinary) cone coincides with the one in the 2-categorical sense, hence it suffices to prove that given two maps $f,g:\mathcal{D}\to \prod \mathcal{C}_i$, whose projections $p_i\circ f$, $p_i \circ g$ are naturally isomorphic (shown by $\eta _i$), the original maps $f$, $g$ are also isomorphic (just take $\eta _i$ in the $i^{\text{th}}$ coordinate) and the isomorphism is uniquely determined by the $\eta _i$-s.

Now we show the existence of all homotopy pullbacks. By the existence of the Joyal model structure on $\mathbf{Cat}$ we can factor any functor $f:\mathcal{C}\to \mathcal{D}$ as $\mathcal{C}\xrightarrow{j}\mathcal{C}'\xrightarrow{f'}\mathcal{D}$ where $j$ is an equivalence (and it is injective on objects), and $f'$ is an isofibration (see e.g.~Theorem 6.2. in \cite{joyal}). Obviously each equivalence with coherent domain is a coherent functor with coherent codomain, and if $f$, $\mathcal{C}$, $\mathcal{D}$ are coherent, then so are $\mathcal{C}'$, $j$ and $f'$. As the (1-)pullback along an isofibration exists (see \ref{cohpb}), we have a natural candidate for a homotopy pullback.

Assume that we are given the coherent functors $f:\mathcal{B}\to \mathcal{A}$, $g^*:\mathcal{C}\to \mathcal{A}$, we would like to form their 2-pullback. Factor $g^*$ as $\mathcal{C}\xrightarrow{\varphi } \mathcal{C}'\xrightarrow{g} \mathcal{A}$, where $\varphi $ is an equivalence and $g$ is an isofibration. Now form the 1-pullback of $f$ and $g$ in $\mathbf{Coh}$, then precompose $f'$ to the quasi-inverse $\varphi ^*$ of $\varphi$ to get the edge of a cone, whose codomain is $\mathcal{C}$. We claim that this construction results a 2-pullback in the 2-category $\mathbf{Coh_{\sim}}$.

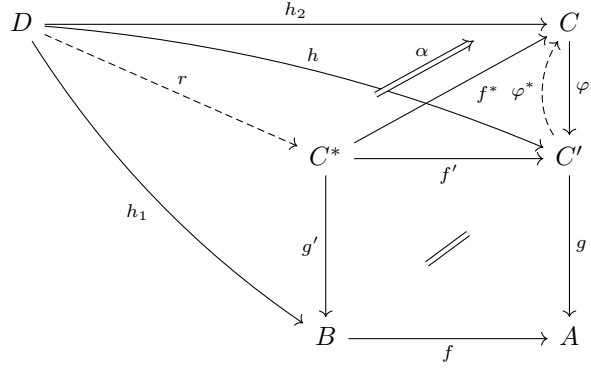
\begin{center}
\begin{tikzcd}
	D &&&&&&& C \\
	\\
	&&&& {C^*} &&& {C'} \\
	\\
	\\
	&&&& B &&& A
	\arrow["{f'}"', from=3-5, to=3-8]
	\arrow["{g'}"', from=3-5, to=6-5]
	\arrow["f"', from=6-5, to=6-8]
	\arrow["g", from=3-8, to=6-8]
	\arrow["\varphi", from=1-8, to=3-8]
	\arrow["{f^*}"'{pos=0.6}, from=3-5, to=1-8]
	\arrow["{\varphi ^*}", curve={height=-12pt}, dashed, from=3-8, to=1-8]
	\arrow["{h_1}"', curve={height=12pt}, from=1-1, to=6-5]
	\arrow["{h_2}", from=1-1, to=1-8]
	\arrow["r", dashed, from=1-1, to=3-5]
	\arrow["h", curve={height=-12pt}, from=1-1, to=3-8]
	\arrow[shorten <=37pt, shorten >=37pt, Rightarrow, no head, from=6-5, to=3-8]
	\arrow["\alpha", shift left=5, shorten <=16pt, shorten >=25pt, Rightarrow, from=3-5, to=1-8]
\end{tikzcd}
\end{center}

First, it is clear that $g\varphi f^* \cong fg'$ and the components of this isomorphism ($\eta $) are given by: $$\eta _x: fg'(x)=gf'(x)\xrightarrow{g(\chi f')}g\varphi \varphi ^* f' (x)= g\varphi f^*(x)$$ where $\chi : 1_{C'}\to \varphi \varphi ^*$ is the unit of the (adjoint) equivalence $\varphi$.

Assume that we are given the maps $h_1: D\to B$, $h_2:D\to C$ and an isomorphism $\nu : fh_1\Rightarrow g\varphi h_2$. The isomorphisms $\nu _d:fh_1(d)\to g\varphi h_2$ can be lifted to $C'$ (as $g$ is an isofibration), hence there are isomorphisms $\mu _d: c'_d\to \varphi h_2 (d)$ with $g(\mu _d)=\nu _d$. We define a functor $h:D\to C'$: it takes an object $d$ to $c'_d$ and an arrow $i:d\to d'$ to $\mu _{d'}^{-1} \circ \varphi h_2(i) \circ \mu _d$.

\[\begin{tikzcd}
	{c'_d} && {\varphi h_2(d)} \\
	\\
	{c'_{d'}} && {\varphi h_2(d')}
	\arrow["{\mu_d}", from=1-1, to=1-3]
	\arrow[from=1-1, to=3-1]
	\arrow["{\varphi h_2(i)}", from=1-3, to=3-3]
	\arrow["{\mu _{d'}}"', from=3-1, to=3-3]
\end{tikzcd}\]

It is clear that $\mu $ is a natural isomorphism from $h$ to $\varphi h_2$. It also follows that $gh=fg'$ (just take the image of the above square at $g$). By the universal property of the pullback we have an arrow $r:D\to C^*$ with $g'r=h_1$ and $f'r=h$.

We need to find a natural isomorphism $\alpha : f^*r\Rightarrow h_2$ such that the composite $$fh_1=fg'r\xRightarrow{\eta r} g\varphi f^* r \xRightarrow{g\varphi (\alpha )} g\varphi h_2$$ gives $\nu$. We will take $$\alpha :=f^*r = \varphi ^*f'r =\varphi ^* h \xRightarrow{\varphi ^*(\mu )}\varphi ^* \varphi h_2 \xRightarrow{\chi 'h_2}h_2$$ (where $\chi ': \varphi ^*\varphi \to 1_C$ is the counit) and then check the above property. The solution is shipped by

\[\begin{tikzcd}
	{f'r(d)} && {\varphi \varphi ^*f'r(d)} &&& {\varphi \varphi ^*\varphi h_2(d)} && {\varphi h_2(d)} \\
	\\
	&& {f'r(d)} &&& {\varphi h_2(d)} \\
	&& {}
	\arrow["{\chi ^{-1}_{f'r(d)}}", from=1-3, to=3-3]
	\arrow["{\varphi \varphi ^*(\mu _d)}", from=1-3, to=1-6]
	\arrow["{\mu _d}"', from=3-3, to=3-6]
	\arrow["{\chi ^{-1}_{\varphi h_2(d)}}"', from=1-6, to=3-6]
	\arrow["{\chi _{f'r(d)}}", from=1-1, to=1-3]
	\arrow[Rightarrow, no head, from=1-1, to=3-3]
	\arrow["{\varphi (\chi '_{h_2(d)})}", from=1-6, to=1-8]
	\arrow[Rightarrow, no head, from=3-6, to=1-8]
\end{tikzcd}\]
where the right triangle commutes as $\chi $ and $\chi '$ satisfy the triangle identities.

It remains to prove that $r$ is unique up to unique compatible natural isomorphism. Assume that we are given
\[\begin{tikzcd}
	D \\
	&& {C^*} && C \\
	&&&& {C'} \\
	&& B && A
	\arrow["{g'}", from=2-3, to=4-3]
	\arrow["f"', from=4-3, to=4-5]
	\arrow["{f^*}"', from=2-3, to=2-5]
	\arrow["{r_2}", shift left=2, from=1-1, to=2-3]
	\arrow["{r_1}"', shift right=2, from=1-1, to=2-3]
	\arrow[""{name=0, anchor=center, inner sep=0}, "{h_2}", curve={height=-18pt}, from=1-1, to=2-5]
	\arrow[""{name=1, anchor=center, inner sep=0}, "{h_1}"', curve={height=12pt}, from=1-1, to=4-3]
	\arrow["g", from=3-5, to=4-5]
	\arrow["\varphi", from=2-5, to=3-5]
	\arrow["{\eta }", shorten <=11pt, shorten >=11pt, Rightarrow, from=4-3, to=2-5]
	\arrow["{\alpha '_1,\alpha '_2}"', shorten >=3pt, Rightarrow, from=2-3, to=0]
	\arrow["{\beta _1, \beta _2}", shorten >=7pt, Rightarrow, from=2-3, to=1]
\end{tikzcd}\]
such that the composition of the 2-cells $\alpha '_i$, $\beta _i^{-1}$ and $\eta $ is $\nu$. We first show that this implies the case
\[\begin{tikzcd}
	&&&& C \\
	D \\
	&& {C^*} && {C'} \\
	\\
	&& B && A
	\arrow["{f'}"', from=3-3, to=3-5]
	\arrow["g", from=3-5, to=5-5]
	\arrow["{g'}", from=3-3, to=5-3]
	\arrow["f"', from=5-3, to=5-5]
	\arrow["{r_1,r_2}"{description, pos=0.4}, from=2-1, to=3-3]
	\arrow[""{name=0, anchor=center, inner sep=0}, "{h_1}"', curve={height=6pt}, from=2-1, to=5-3]
	\arrow[""{name=1, anchor=center, inner sep=0}, "h", curve={height=-12pt}, from=2-1, to=3-5]
	\arrow["{f^*}"'{pos=0.7}, from=3-3, to=1-5]
	\arrow["{\varphi }", from=1-5, to=3-5]
	\arrow["{h_2}", from=2-1, to=1-5]
	\arrow[shorten <=22pt, shorten >=22pt, Rightarrow, no head, from=5-3, to=3-5]
	\arrow["{\varphi ^*}"', curve={height=18pt}, dashed, from=3-5, to=1-5]
	\arrow["{\beta _1,\beta _2}"', shorten >=6pt, Rightarrow, from=3-3, to=0]
	\arrow["{\alpha _1,\alpha _2}", shorten >=3pt, Rightarrow, from=3-3, to=1]
\end{tikzcd}\]
with $f(\beta _i)=g(\alpha _i)$.

We take $$\alpha _i:= f'r_i\xRightarrow{\chi _{f'r_i}} \varphi \varphi ^*f'r_i \xRightarrow{\varphi (\alpha _i')} \varphi h_2 \xRightarrow{\mu ^{-1}} h$$. The fact that $\eta$, $\beta _i^{-1}$ and $\alpha _i$ glue together and form $\nu$ can be expressed by the commutative square:

\[\begin{tikzcd}
	{fg'r_i(d)} && {g\varphi f^*r_i(d)} \\
	\\
	{fh_1(d)} && {g\varphi h_2(d)}
	\arrow["{\eta _{r_i(d)}}", from=1-1, to=1-3]
	\arrow["{g\varphi ((\alpha _i')_d)}", from=1-3, to=3-3]
	\arrow["{f((\beta_i)_d)}"', from=1-1, to=3-1]
	\arrow["{\nu _d}"', from=3-1, to=3-3]
\end{tikzcd}\]
hence $f((\beta _i)_d)$ is $g(\mu _d^{-1})\circ g\varphi ((\alpha '_i)_d) \circ g(\chi _{f'r(d)})=g((\alpha _i)_d)$.

Now if we take $$\alpha :=\alpha _2^{-1}\alpha _1: f'r_1 \Rightarrow f'r_2$$ and $$\beta := \beta _2^{-1}\beta _1: g'r_1\Rightarrow g'r_2$$, then we have $f(\beta )=g(\alpha )$.

Recall that the (coherent) category $C^*$ can be explicitly described as the one with objects $\{(b,c'): b\in Ob(B), c'\in Ob(C'), f(b)=g(c')\}$ and similarly for arrows. By the above property $(\beta , \alpha )$ is a well-defined natural isomorphism $r_1=(g'r_1,f'r_1)\Rightarrow (g'r_2,f'r_2)=r_2$ and it is uniquely determined by the 2-cells $\alpha _i$, $\beta _i$, $\eta $ and $\nu $. But the latter is equivalent to the datum of $\alpha '_i$, $\beta _i$, $\eta $ and $\nu $. 

\end{proof}

Recall the following Proposition (4.4.2.6.) together with its dual from \cite{lurietopos}.

\begin{theorem}[Lurie]
If $\mathcal{C}$ is an $(\infty , 1)$-category and it has (homotopy) pushouts and $\kappa$-small (homotopy) coproducts then $\mathcal{C}$ has all $\kappa$-small (homotopy) colimits.
\end{theorem}

\begin{corollary}
$\mathbf{Coh_{\sim}}$ has all homotopy limits.
\end{corollary}

Now we would like to prove the existence of homotopy colimits. First we need the basic fact that equalizers are monic:

\begin{proposition}
Given a diagram formed by a set of paralel arrows and some natural isomorphisms between them, its homotopy limit (equalizer)
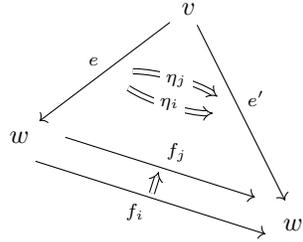
\[\begin{tikzcd}
	&& v \\
	\\
	w \\
	&&& {w'}
	\arrow[""{name=0, anchor=center, inner sep=0}, "{f_i}"', shift right=3, from=3-1, to=4-4]
	\arrow[""{name=1, anchor=center, inner sep=0}, "{f_j}", shift left=2, shorten <=8pt, shorten >=8pt, from=3-1, to=4-4]
	\arrow[""{name=2, anchor=center, inner sep=0}, "e"', from=1-3, to=3-1]
	\arrow[""{name=3, anchor=center, inner sep=0}, "{e'}", from=1-3, to=4-4]
	\arrow[shorten <=1pt, shorten >=1pt, Rightarrow, from=0, to=1]
	\arrow["{\eta _i}"{description}, curve={height=6pt}, shorten <=10pt, shorten >=10pt, Rightarrow, from=2, to=3]
	\arrow["{\eta _j}"{description}, curve={height=-6pt}, shorten <=10pt, shorten >=10pt, Rightarrow, from=2, to=3]
\end{tikzcd}\]
has the property, that for any natural isomorphism $\alpha : eg\Rightarrow eh$ there is a unique natural isomorphism $\gamma : g\Rightarrow h$ such that $e\gamma =\alpha $.
\label{eqmono}
\end{proposition}

\begin{proof}
Take $\beta _i$ to be $e'g\xRightarrow{\eta _i g} f_ieg \xRightarrow{f_i \alpha} f_ieh \xRightarrow{\eta _i^{-1}h} e'h$. Then
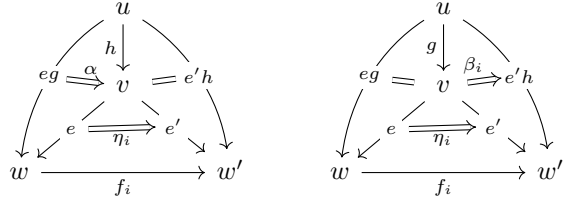
\[\begin{tikzcd}
	& u &&& u \\
	& v &&& v \\
	w && {w'} & w && {w'}
	\arrow[""{name=0, anchor=center, inner sep=0}, "e"{description}, from=2-2, to=3-1]
	\arrow[""{name=1, anchor=center, inner sep=0}, "{e'}"{description}, from=2-2, to=3-3]
	\arrow["{f_i}"', from=3-1, to=3-3]
	\arrow[""{name=2, anchor=center, inner sep=0}, "eg"{description}, curve={height=12pt}, from=1-2, to=3-1]
	\arrow[""{name=3, anchor=center, inner sep=0}, "{e'h}"{description}, curve={height=-12pt}, from=1-2, to=3-3]
	\arrow["h"', from=1-2, to=2-2]
	\arrow["{f_i}"', from=3-4, to=3-6]
	\arrow[""{name=4, anchor=center, inner sep=0}, "e"{description}, from=2-5, to=3-4]
	\arrow[""{name=5, anchor=center, inner sep=0}, "{e'}"{description}, from=2-5, to=3-6]
	\arrow[""{name=6, anchor=center, inner sep=0}, "eg"{description}, curve={height=12pt}, from=1-5, to=3-4]
	\arrow[""{name=7, anchor=center, inner sep=0}, "{e'h}"{description}, curve={height=-12pt}, from=1-5, to=3-6]
	\arrow["g"', from=1-5, to=2-5]
	\arrow["{\eta _i}"', shorten <=6pt, shorten >=6pt, Rightarrow, from=0, to=1]
	\arrow["\alpha", shorten <=6pt, Rightarrow, from=2, to=2-2]
	\arrow[shorten <=4pt, shorten >=8pt, Rightarrow, no head, from=2-2, to=3]
	\arrow["{\eta _i}"', shorten <=6pt, shorten >=6pt, Rightarrow, from=4, to=5]
	\arrow[shorten <=8pt, shorten >=4pt, Rightarrow, no head, from=6, to=2-5]
	\arrow["\beta _i", shorten <=2pt, shorten >=6pt, Rightarrow, from=2-5, to=7]
\end{tikzcd}\]
are both splittings of the 2-cells $f_ieg \xRightarrow{f_i\alpha} f_ieh \xRightarrow{\eta _i h} e'h$, hence there is a unique natural isomorphism $\gamma :g\Rightarrow h$ for which $\alpha =e\gamma $ and $\beta _i =e'\gamma$. The latter is easily proved to be redundant.
\end{proof}

The following is the 2-categorical analogue of Theorem V.6.1. in \cite{maclane}.

\begin{proposition}
Let $\mathcal{C}$ be a locally small 2-complete strict (2,1)-category. Assume that there exists a small set $X\subset Ob(\mathcal{C})$ such that for each $c\in Ob(\mathcal{C})$ there is an element $x\in X$ and an arrow $x\to c$. Then $\mathcal{C}$ has a homotopy initial object. 
\end{proposition}

\begin{proof}
The product $w=\prod _{x\in X} x$ is a weak initial object, i.e.~given any other object $c$ there is at least one map $w\to c$ (e.g.~the assumed one composed with the suitable projection). By assumption the class $Hom(w,w)$ is a set, hence we can take its joint 2-equalizer: the homotopy limit of this 1-dimensional diagram:
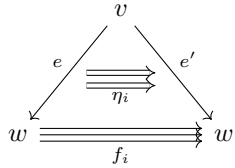
\[\begin{tikzcd}
	& v \\
	\\
	w && w
	\arrow[from=3-1, to=3-3]
	\arrow[shift left=1, from=3-1, to=3-3]
	\arrow["{f_i}"', shift right=1, from=3-1, to=3-3]
	\arrow[""{name=0, anchor=center, inner sep=0}, "e"', from=1-2, to=3-1]
	\arrow[""{name=1, anchor=center, inner sep=0}, "{e'}", from=1-2, to=3-3]
	\arrow[shorten <=6pt, shorten >=6pt, Rightarrow, from=0, to=1]
	\arrow["{\eta _i}"', shift right=2, shorten <=6pt, shorten >=6pt, Rightarrow, from=0, to=1]
\end{tikzcd}\]
We claim that $v$ is homotopy initial. It is clear, that $v$ is a weakly initial object. Let $g,h:v\to c$ be two maps. By assumption the natural isomorphisms between $f$ and $g$ form a set, and we can take the equalizer of this (now 2-dimensional) diagram.
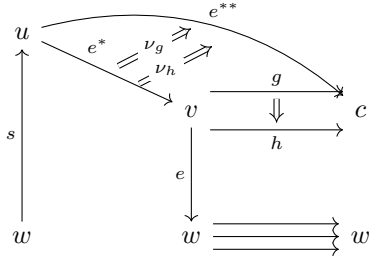
\[\begin{tikzcd}
	u \\
	&& v && c \\
	\\
	w && w && w
	\arrow[""{name=0, anchor=center, inner sep=0}, "h"', shift right=3, from=2-3, to=2-5]
	\arrow[""{name=1, anchor=center, inner sep=0}, "g", shift left=3, from=2-3, to=2-5]
	\arrow[shift right=2, from=4-3, to=4-5]
	\arrow[shift left=2, from=4-3, to=4-5]
	\arrow[from=4-3, to=4-5]
	\arrow["e"', from=2-3, to=4-3]
	\arrow[""{name=2, anchor=center, inner sep=0}, "{e^*}"{pos=0.3}, from=1-1, to=2-3]
	\arrow[""{name=3, anchor=center, inner sep=0}, "{e^{**}}", curve={height=-24pt}, from=1-1, to=2-5]
	\arrow["s", from=4-1, to=1-1]
	\arrow[shorten <=2pt, shorten >=2pt, Rightarrow, from=1, to=0]
	\arrow["{\nu _g}"{description}, shorten <=4pt, shorten >=4pt, Rightarrow, from=2, to=3]
	\arrow["{\nu _h}"{description}, shift right=4, shorten <=8pt, Rightarrow, from=2, to=3]
\end{tikzcd}\]
As $w$ was weakly initial there is a map $s:w\to u$. Since $e$ was the equalizer of all homomorphisms from $w$ to $w$, there is a natural isomorphism: $ee^*se \xRightarrow{\eta _{ee^*s}} e' \xRightarrow{\eta _{1_w}^{-1}} e$, and by Proposition \ref{eqmono} there is a unique isomorphism $\gamma : e^*(se)\Rightarrow 1_v$ such that $e\gamma =\eta _{1_w}^{-1}\eta _{ee^*s}$ (but this fact will not be used).

We can construct a natural isomorphism: $$g\xRightarrow{g\gamma ^{-1}} ge^*(se) \xRightarrow{\nu _g(se)} e^{**}(se) \xRightarrow{\nu _h^{-1}(se)} he^*(se)\xRightarrow{h\gamma }h $$ For uniqueness we have to prove that the pentagon

\[\begin{tikzcd}
	&& {e^{**}(se)} \\
	{ge^*(se)} &&&& {he^*(se)} \\
	g &&&& h
	\arrow["\chi"', Rightarrow, from=3-1, to=3-5]
	\arrow["{\chi (e^*se)}"', Rightarrow, from=2-1, to=2-5]
	\arrow["g\gamma", Rightarrow, from=2-1, to=3-1]
	\arrow["h\gamma"', Rightarrow, from=2-5, to=3-5]
	\arrow["{\nu_g (se)}", Rightarrow, from=2-1, to=1-3]
	\arrow["{\nu_h^{-1}(se)}", Rightarrow, from=1-3, to=2-5]
\end{tikzcd}\]
commutes for arbitrary $\chi$. The first floor commutes by the interchange law for strict 2-categories (as both composites must be equal to the horizontal composite of $\gamma $ and $\chi $), the roof commutes as $(u,e^*,e^{**})$ form a homotopy cone, in particular $\chi e^*=\nu _h^{-1}\nu _g$.

\end{proof}

Recall that given a map $p:K\to \mathcal{C}$ of simplicial sets (where $\mathcal{C}$ is an $(\infty ,1)$-category), its (homotopy) colimit is an initial object of $\mathcal{C}_{p/}$ (the infinity category of homotopy cocones, or the undercategory) (See: \cite{lurietopos}). Therefore we would like to use our previous statement for $\mathcal{C}_{p/}$. The following theorem is due to Pál Zsámboki.

\begin{theorem}
Let $\mathcal{C}$ be a complete $(\infty ,1)$-category and let $p:K\to \mathcal{C}$ be a map of simplicial sets. Then the undercategory $\mathcal{C}_{p/}$ is also complete.
\end{theorem}

\begin{proof}
We will use that the construction $K\star L$ gives a simplicial set with the property that given a simplicial map $q:L\to \mathcal{C}$, the simplicial maps of the form $K\to \mathcal{C}_{/q}$ are the same as those maps $K\star L \to \mathcal{C}$ whose restriction to $L$ gives $q$ (and the dual property holds for $\mathcal{C}_{p/}$).

Let $L$ be a simplicial set and $L\xrightarrow{q} \mathcal{C}_{p/}$ be a diagram. Then $q$ is a diagram $K\star L \to \mathcal{C}$. Let $q_0:L\to \mathcal{C}$ be its restriction and $\bar{q_0}:\Delta _0 \star L\to \mathcal{C}$ be the limit for $q_0$. That is, the restriction map $\mathcal{C}_{/\bar{q_0}} \to \mathcal{C}_{/q_0}$ is a trivial fibration, in particular there is a lift in

\[\begin{tikzcd}
	\emptyset && {\mathcal{C}_{/\bar{q_0}}} \\
	\\
	K && {\mathcal{C}_{/q_0}}
	\arrow[from=1-1, to=3-1]
	\arrow[from=1-1, to=1-3]
	\arrow[from=1-3, to=3-3]
	\arrow["q", from=3-1, to=3-3]
	\arrow["{\bar{q}}"{description}, dashed, from=3-1, to=1-3]
\end{tikzcd}\]
Then $\bar{q}$ corresponds to a map $K\star \Delta _0 \star L \to \mathcal{C}$, i.e.~to a map $\Delta _0 \star L \to \mathcal{C}_{p/}$. We claim that $\bar{q}$ is the limit of $q$, that is the restriction map $(\mathcal{C} _{p/})_{/\bar{q_0}} \to (\mathcal{C} _{p/})_{/q_0}$ is a trivial fibration. Let $X\hookrightarrow X'$ be an inclusion of simplicial sets. The lifting problem

\[\begin{tikzcd}
	X && {(\mathcal{C} _{p/})_{/\bar{q_0}}} \\
	\\
	{X'} && {(\mathcal{C} _{p/})_{/q_0}}
	\arrow[hook, from=1-1, to=3-1]
	\arrow[from=3-1, to=3-3]
	\arrow[from=1-3, to=3-3]
	\arrow[from=1-1, to=1-3]
	\arrow[dashed, from=3-1, to=1-3]
\end{tikzcd}\]
is the same as the lifting problem

\[\begin{tikzcd}
	{K\star X} && {\mathcal{C} _{/\bar{q_0}}} \\
	\\
	{K\star X'} && {\mathcal{C}_{/q_0}}
	\arrow[hook, from=1-1, to=3-1]
	\arrow[from=3-1, to=3-3]
	\arrow[from=1-3, to=3-3]
	\arrow[from=1-1, to=1-3]
	\arrow[dashed, from=3-1, to=1-3]
\end{tikzcd}\]
and thus has a solution as $\bar{q_0}$ was the limit of $q_0$.
\end{proof}

\begin{corollary}
The locally small strict (2,1)-category $(\mathbf{Coh_{\sim}})_{p/}$ is 2-complete. (In particular it is non-empty by the existence of a terminal object.)
\end{corollary}

It remains to find a weakly initial family of homotopy cocones over $p$. It is a straightforward consequence of Proposition \ref{solution}: given a homotopy cocone (with top object $c$), the joint image of the edges $p(i)\to c$ ($i\in K_0$) is included in some coherent subcategory with cardinality $\leq \aleph _0 \cdot \prod _{i\in K_0} |p(i)|$, hence the set of all cocones with top object having at most this cardinality is a solution set. We proved:

\begin{theorem}
The (2,1)-category $\mathbf{Coh_{\sim}}$ is 2-complete and 2-cocomplete.
\end{theorem}

\section{Small object argument}

In this section we will generalise the classical small object argument for locally small 2-cocomplete strict (2,1)-categories, which we will typically denote by $\mathbf{C}$. The proof follows the one given in \cite{hovey} for the 1-categorical setting.

\begin{definition}
Given a 2-colimit preserving diagram $\lambda \to \mathbf{C}$ with homotopy colimit $\mathcal{X}$
\[\begin{tikzcd}
	&& {\mathcal{X}} \\
	{\mathcal{X}_0} & {\mathcal{X}_1} & {\mathcal{X}_2} & \dots
	\arrow["{f_0}"', from=2-1, to=2-2]
	\arrow["{f_1}"', from=2-2, to=2-3]
	\arrow["{f_2}"', from=2-3, to=2-4]
	\arrow["f", from=2-1, to=1-3]
	\arrow[from=2-2, to=1-3]
	\arrow[from=2-3, to=1-3]
\end{tikzcd}\]
the coprojection map $f:\mathcal{X}_0\to \mathcal{X}$ is called the \emph{transfinite composition} of the $\lambda $-sequence $(f_i)_{i<\lambda }$.
\end{definition}

\begin{definition}
Let $I\subset Arr(\mathbf{C})$ be a set. \emph{$I$-cell} is the class of maps that can be written as the transfinite composition of 2-pushouts from $I$. \emph{$I$-inj} is the class whose members ($f$) have the following right lifting property: given a square
\[\begin{tikzcd}
	{\mathcal{C}} && {\mathcal{X}} \\
	\\
	{\mathcal{D}} && {\mathcal{Y}}
	\arrow["g"', from=1-1, to=3-1]
	\arrow["{h'}", from=3-1, to=3-3]
	\arrow["f", from=1-3, to=3-3]
	\arrow["h", from=1-1, to=1-3]
	\arrow["\eta"{description}, shorten <=11pt, shorten >=11pt, Rightarrow, from=3-1, to=1-3]
\end{tikzcd}\]
with $g\in I$, there is a lifting
\[\begin{tikzcd}
	{\mathcal{C}} && {\mathcal{X}} \\
	\\
	{\mathcal{D}} && {\mathcal{Y}}
	\arrow["g"', from=1-1, to=3-1]
	\arrow[""{name=0, anchor=center, inner sep=0}, "{h'}", from=3-1, to=3-3]
	\arrow["f", from=1-3, to=3-3]
	\arrow[""{name=1, anchor=center, inner sep=0}, "h", from=1-1, to=1-3]
	\arrow["k"{description}, from=3-1, to=1-3]
	\arrow["{\nu_1}"{description}, shorten <=13pt, shorten >=9pt, Rightarrow, from=3-1, to=1]
	\arrow["{\nu_2}"{description}, shorten <=13pt, shorten >=9pt, Rightarrow, from=0, to=1-3]
\end{tikzcd}\]
such that the composition of $\nu _1$ and $\nu _2$ is $\eta$.

\emph{$I$-proj} is the class whose members have the left lifting property wrt.~$I$. As usual \emph{$I$-cof}=($I$-inj)-proj, and \emph{$I$-fib}=($I$-proj)-inj.
\end{definition}

\begin{proposition}
$I$-cell $\subseteq I$-cof.
\end{proposition}

\begin{proof}
Clearly $I\subseteq I$-cof, hence it suffices to prove that $I$-cof is closed under pushouts and transfinite compositions. First we show that if $f$ has the left lifting property wrt. $m$, then its 2-pushout $f'$ has also.

$l_1$ is induced by the lifting property of $f$ and $l_2$ by the universality of the 2-pushout.
\[\begin{tikzcd}
	\bullet && \bullet && \bullet \\
	\\
	\bullet && \bullet && \bullet
	\arrow[""{name=0, anchor=center, inner sep=0}, "f"{description}, from=1-1, to=3-1]
	\arrow["{g'}"{description}, from=3-1, to=3-3]
	\arrow[""{name=1, anchor=center, inner sep=0}, "g"{description}, from=1-1, to=1-3]
	\arrow["{f'}"{description}, from=1-3, to=3-3]
	\arrow["h"{description}, from=1-3, to=1-5]
	\arrow[""{name=2, anchor=center, inner sep=0}, "m"{description}, from=1-5, to=3-5]
	\arrow[""{name=3, anchor=center, inner sep=0}, "k"{description}, from=3-3, to=3-5]
	\arrow["{l_1}"{description, pos=0.7}, curve={height=12pt}, dashed, from=3-1, to=1-5]
	\arrow["{l_2}"{description, pos=0.7}, curve={height=-12pt}, dashed, from=3-3, to=1-5]
	\arrow["\alpha"{description}, shorten <=7pt, shorten >=7pt, Rightarrow, from=3, to=2]
	\arrow["\eta"{description}, shorten <=7pt, shorten >=7pt, Rightarrow, from=0, to=1]
\end{tikzcd}\]
The properties that $l_1$ and $l_2$ are splittings of the related 2-cells can be written as $\eta +\alpha =\gamma + \beta ^{-1}$ and $\eta +\nu +\mu =\gamma $.

We should prove that $l_2$ is a splitting of $\alpha $. It is enough to see that in
\[\begin{tikzcd}
	\bullet && \bullet &&&& \bullet && \bullet \\
	\\
	\bullet && \bullet &&&& \bullet && \bullet && \bullet \\
	&&& \bullet \\
	&&&& \bullet &&&& \bullet && \bullet
	\arrow["g"{description}, from=1-1, to=1-3]
	\arrow["{f'}"{description}, from=1-3, to=3-3]
	\arrow["f"{description}, from=1-1, to=3-1]
	\arrow["{g'}"{description}, from=3-1, to=3-3]
	\arrow[""{name=0, anchor=center, inner sep=0}, "h"{description}, curve={height=-12pt}, from=1-3, to=4-4]
	\arrow[""{name=1, anchor=center, inner sep=0}, "{l_1}"{description}, curve={height=12pt}, from=3-1, to=4-4]
	\arrow["m"{description}, from=4-4, to=5-5]
	\arrow[""{name=2, anchor=center, inner sep=0}, "{l_2}"{description}, from=3-3, to=4-4]
	\arrow["\eta"{description}, shorten <=17pt, shorten >=17pt, Rightarrow, from=3-1, to=1-3]
	\arrow["f"{description}, from=1-7, to=3-7]
	\arrow["{g'}"{description}, from=3-7, to=3-9]
	\arrow["g"{description}, from=1-7, to=1-9]
	\arrow["{f'}"{description}, from=1-9, to=3-9]
	\arrow["m"{description}, from=3-11, to=5-11]
	\arrow[""{name=3, anchor=center, inner sep=0}, "h"{description}, from=1-9, to=3-11]
	\arrow[""{name=4, anchor=center, inner sep=0}, "k"{description}, from=3-9, to=5-11]
	\arrow[""{name=5, anchor=center, inner sep=0}, "{l_1}"{description}, from=3-7, to=5-9]
	\arrow["m"{description}, from=5-9, to=5-11]
	\arrow["\eta"{description}, shorten <=17pt, shorten >=17pt, Rightarrow, from=3-7, to=1-9]
	\arrow["\mu", shorten <=5pt, shorten >=5pt, Rightarrow, from=2, to=0]
	\arrow["\nu", shorten <=7pt, shorten >=7pt, Rightarrow, from=1, to=2]
	\arrow["\alpha"{description}, shorten <=9pt, shorten >=9pt, Rightarrow, from=4, to=3]
	\arrow["\beta"{description}, shorten <=13pt, shorten >=13pt, Rightarrow, from=5, to=4]
\end{tikzcd}\]
the 2-cells filling the boundaries are identical as in this case both $ml_2$ and $k$ are suitable splittings, hence there is a unique natural isomorphism $\delta:k \Rightarrow ml_2$ for which $\beta +\delta =m\nu $ and $\mu +\delta =\alpha$. This follows from the identities observed above.

Now assume that each $f_i$ ($i<\lambda $) has left lifting property wrt.~$m$ (and that $f_i$-s form a (co)continuous sequence). We have to prove that its transfinite composition $f$ has the same lifting property. The proof is similar to the previous one and it is pictured as
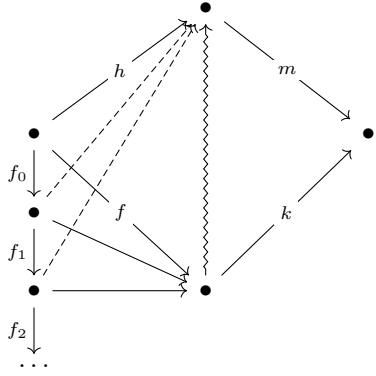
\[\begin{tikzcd}
	&& \bullet \\
	\\
	\bullet &&&& \bullet \\
	\bullet \\
	\bullet && \bullet \\
	\dots
	\arrow["{f_0}"', from=3-1, to=4-1]
	\arrow["{f_1}"', from=4-1, to=5-1]
	\arrow["{f_2}"', from=5-1, to=6-1]
	\arrow["f"{description}, from=3-1, to=5-3]
	\arrow[from=4-1, to=5-3]
	\arrow[from=5-1, to=5-3]
	\arrow["h"{description}, from=3-1, to=1-3]
	\arrow["m"{description}, from=1-3, to=3-5]
	\arrow["k"{description}, from=5-3, to=3-5]
	\arrow[dashed, from=4-1, to=1-3]
	\arrow[dashed, from=5-1, to=1-3]
	\arrow[squiggly, from=5-3, to=1-3]
\end{tikzcd}\]
\end{proof}

It is worth to write out explicitly:
\begin{proposition}
(Homotopy) left lifting properties are preserved by (homotopy) pushouts and transfinite compositions. Dually, right lifting properties are preserved by pullbacks and transfinite cocompositions (homotopy limit of the reversed sequence). In particular $I$-$inj$ and $I$-$proj$ are subcategories.
\end{proposition}

\begin{proposition}
$I$-cell is closed under transfinite composition.
\end{proposition}

\begin{proof}
We need to prove that "the transfinite composition of transfinite compositions is a transfinite composition", i.e.~that if we have a sequential (2-)diagram then its colimit can be computed as the colimit of any cofinal subsequence. This is Proposition 4.1.1.8. in \cite{lurietopos}.
\end{proof}

\begin{proposition}
The homotopy pushout of a coproduct of maps from $I$ is in $I$-cell.
\end{proposition}

\begin{proof}
Let $g_j$ $(j\in J)$ be a family of arrows from $I$. Their coproduct is the induced map:
\[\begin{tikzcd}
	{\mathcal{C}_j} && {\cup _{j\in J} \mathcal{C}_j} \\
	\\
	{\mathcal{D}_j} && {\cup _{j\in J}\mathcal{D}_j}
	\arrow["{g_j}", from=1-1, to=3-1]
	\arrow[from=3-1, to=3-3]
	\arrow[from=1-1, to=1-3]
	\arrow["{\cup g_j}", dashed, from=1-3, to=3-3]
	\arrow["{\chi _j}"{description}, shorten <=11pt, shorten >=11pt, Rightarrow, from=1-3, to=3-1]
\end{tikzcd}\]
Now take the 2-pushout:
\[\begin{tikzcd}
	{\cup _j\mathcal{C}_j} && {\mathcal{X}} \\
	\\
	{\cup _j\mathcal{D}_j} && {\mathcal{Y}}
	\arrow["{\cup g_j}"', from=1-1, to=3-1]
	\arrow["{h_1}", from=3-1, to=3-3]
	\arrow["f", from=1-3, to=3-3]
	\arrow["{h_0}", from=1-1, to=1-3]
	\arrow["\eta"{description}, shorten <=11pt, shorten >=11pt, Rightarrow, from=1-3, to=3-1]
\end{tikzcd}\]
We will proceed by transfinite recursion and take: $X_0=X$, $\rho _0=f$ and $i_{0,0}=1_X$. In the successor step we form the 2-pushout of $g_j:\mathcal{C}_j\to \mathcal{D}_j$ and $\mathcal{C}_j\to \cup \mathcal{C}_j \xrightarrow{h_0} \mathcal{X} \xrightarrow{i_{0,j}}\mathcal{X}_j$ to get $X_{j+1}$ and induce $\rho _{j+1}$ by the universal property of the square. Hence we get a commutative cube (where the faces are filled with the obvious 2-cells):
\[\begin{tikzcd}
	{\mathcal{C}_j} &&&& {\mathcal{X}_j} \\
	& {\cup \mathcal{C}_j} & {\mathcal{X}} & {\mathcal{X}_j} \\
	\\
	& {\cup \mathcal{D}_j} && {\mathcal{Y}} \\
	{\mathcal{D}_j} &&&& {\mathcal{X}_{j+1}}
	\arrow["{h_0}", from=2-2, to=2-3]
	\arrow["{i_{0,j}}", from=2-3, to=2-4]
	\arrow["{\rho _j}"{description}, from=2-4, to=4-4]
	\arrow["{\cup g_j}"{description}, from=2-2, to=4-2]
	\arrow["{h_1}"{description}, from=4-2, to=4-4]
	\arrow[from=1-1, to=2-2]
	\arrow[Rightarrow, no head, from=1-5, to=2-4]
	\arrow[from=5-1, to=4-2]
	\arrow[dashed, from=5-5, to=4-4]
	\arrow[from=1-5, to=5-5]
	\arrow[from=1-1, to=1-5]
	\arrow["{g_j}"{description}, from=1-1, to=5-1]
	\arrow[from=5-1, to=5-5]
	\arrow["f"{description}, from=2-3, to=4-4]
\end{tikzcd}\]
When $j$ is a limit ordinal $\mathcal{X}_j$ is given by the transfinite composition
\[\begin{tikzcd}
	&&& {\mathcal{Y}} \\
	&&& {\mathcal{X}_j} \\
	{\mathcal{X}_0} && {\mathcal{X}_1} && \dots
	\arrow["{i_{0,1}}"', from=3-1, to=3-3]
	\arrow["{i_{1,2}}"', from=3-3, to=3-5]
	\arrow["{i_{0,j}}"{description}, from=3-1, to=2-4]
	\arrow["{i_{1,j}}"'{pos=0.7}, from=3-3, to=2-4]
	\arrow[dashed, from=2-4, to=1-4]
	\arrow["{\rho_1}"{description}, curve={height=-6pt}, from=3-3, to=1-4]
	\arrow["{\rho _0}"{description}, curve={height=-6pt}, from=3-1, to=1-4]
\end{tikzcd}\]
(the 3-cells are filled). We claim that with $\lambda =|J|$ the map $\mathcal{X}\to \mathcal{X}_{\lambda}$ is also a homotopy pushout for $\cup g_j$ along $h_0$. To see this we should find some 2-cells for
\[\begin{tikzcd}
	{\cup \mathcal{C}_j} &&& {\mathcal{X}} \\
	\\
	{\cup \mathcal{D}_j} &&& {\mathcal{X}_\lambda} \\
	&&&& {\mathcal{Y}}
	\arrow["{h_0}"{description}, from=1-1, to=1-4]
	\arrow["{i_{0,\lambda}}"{description}, from=1-4, to=3-4]
	\arrow["{\cup g_j}"{description}, from=1-1, to=3-1]
	\arrow["{\small{\cup \{\mathcal{D}_j\to \mathcal{X}_{j+1}\to \mathcal{X}_\lambda\}}}", from=3-1, to=3-4]
	\arrow["{\rho _{\lambda}}", from=3-4, to=4-5]
	\arrow["f"{description}, curve={height=-12pt}, from=1-4, to=4-5]
	\arrow["{h_1}"{description}, curve={height=12pt}, from=3-1, to=4-5]
\end{tikzcd}\]
whose composite is $\eta$. They can be found on the surface of the commutative 3-simplicial set
\[\begin{tikzcd}
	{\mathcal{C}_j} &&&& {\mathcal{X}_j} \\
	& {\cup \mathcal{C}_j} & {\mathcal{X}} & {\mathcal{X}_j} \\
	\\
	& {\cup \mathcal{D}_j} && {\mathcal{Y}} \\
	{\mathcal{D}_j} &&&& {\mathcal{X}_{j+1}} \\
	& {\cup\mathcal{D}_j} && {\mathcal{X}_{\lambda}}
	\arrow["{h_0}", from=2-2, to=2-3]
	\arrow["{i_{0,j}}", from=2-3, to=2-4]
	\arrow["{\rho _j}"{description}, from=2-4, to=4-4]
	\arrow["{\cup g_j}"{description}, from=2-2, to=4-2]
	\arrow["{h_1}"{description}, from=4-2, to=4-4]
	\arrow[from=1-1, to=2-2]
	\arrow[Rightarrow, no head, from=1-5, to=2-4]
	\arrow[from=5-1, to=4-2]
	\arrow[dashed, from=5-5, to=4-4]
	\arrow[from=1-5, to=5-5]
	\arrow[from=1-1, to=1-5]
	\arrow["{g_j}"{description}, from=1-1, to=5-1]
	\arrow[from=5-1, to=5-5]
	\arrow["f"{description}, from=2-3, to=4-4]
	\arrow[Rightarrow, no head, from=4-2, to=6-2]
	\arrow[from=5-1, to=6-2]
	\arrow["{\rho _\lambda}"{description, pos=0.3}, from=6-4, to=4-4]
	\arrow[from=5-5, to=6-4]
	\arrow[from=6-2, to=6-4]
\end{tikzcd}\]
\end{proof}

\begin{definition}
An object $\mathcal{X}$ of $\mathbf{C}$ is $\lambda $-small wrt.~a subcategory $J$ if $\mathbf{C}(\mathcal{X},-)$ commutes with $\lambda $-filtered sequential 2-colimits formed in $J$. $\mathcal{X}$ is small if it is $\lambda $-small for some $\lambda $.
\end{definition}

\begin{theorem}[Small object argument]
Let $I\subset Arr(\mathbf{C})$ be a set, and assume that domains of $I$ are small relative to $I$-cell. Then for any map $f:\mathcal{X}\to \mathcal{Y}$ there are arrows $\mathcal{X}\xrightarrow{f'}\mathcal{Z}\xrightarrow{f''}\mathcal{Y}$ such that $f'\in I$-cell, $f''\in I$-inj and $f'' \circ f'$ is isomorphic to $f$.
\label{smallob}
\end{theorem}

\begin{proof}
We proceed by transfinite recursion and take $\mathcal{Z}_0=\mathcal{X}$, $\rho _0 =f$ and $i_{0,0}=1_{\mathcal{X}}$.

For successor ordinal $j+1$ collect all squares
\[\begin{tikzcd}
	{\mathcal{A}_s} && {\mathcal{Z}_j} \\
	\\
	{\mathcal{B}_s} && {\mathcal{Y}}
	\arrow["{g_s}"{description}, from=1-1, to=3-1]
	\arrow["{h_s}"{description}, from=1-1, to=1-3]
	\arrow["{\rho _j}"{description}, from=1-3, to=3-3]
	\arrow["{k_s}"{description}, from=3-1, to=3-3]
	\arrow["{\eta _s}"{description}, shorten <=11pt, shorten >=11pt, Rightarrow, from=1-3, to=3-1]
\end{tikzcd}\]
with $g_s\in I$ to an $S$-indexed set, then form the 2-pushout of $\sqcup g_s$ and $\sqcup h_s$ and induce $\rho _{j+1}$:

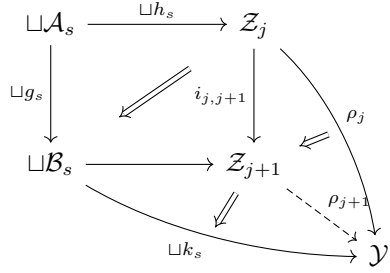
\[\begin{tikzcd}
	{\sqcup \mathcal{A}_s} && {\mathcal{Z}_j} \\
	\\
	{\sqcup \mathcal{B}_s} && {\mathcal{Z}_{j+1}} \\
	&&& {\mathcal{Y}}
	\arrow["{\sqcup h_s}", from=1-1, to=1-3]
	\arrow["{\sqcup g_s}"', from=1-1, to=3-1]
	\arrow[from=3-1, to=3-3]
	\arrow["{i_{j,j+1}}"', from=1-3, to=3-3]
	\arrow[""{name=0, anchor=center, inner sep=0}, "{\rho _j}", curve={height=-12pt}, from=1-3, to=4-4]
	\arrow[""{name=1, anchor=center, inner sep=0}, "{\sqcup k_s}"', curve={height=12pt}, from=3-1, to=4-4]
	\arrow["{\rho _{j+1}}", dashed, from=3-3, to=4-4]
	\arrow[shorten <=17pt, shorten >=17pt, Rightarrow, from=1-3, to=3-1]
	\arrow[shorten <=4pt, shorten >=2pt, Rightarrow, from=0, to=3-3]
	\arrow[shorten <=2pt, shorten >=3pt, Rightarrow, from=3-3, to=1]
\end{tikzcd}\]
Note that the composition of the three 2-cells is the natural isomorphism induced by $\{ \eta _s : s\in S \}$. (*)

When $j$ is a limit ordinal we form the transfinite composition

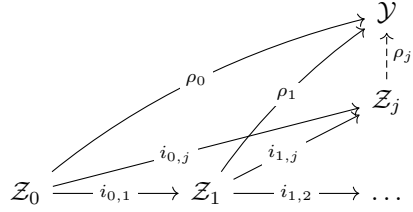
\[\begin{tikzcd}
	&&&& {\mathcal{Y}} \\
	&&&& {\mathcal{Z}_j} \\
	{\mathcal{Z}_0} && {\mathcal{Z}_1} && \dots
	\arrow["{i_{0,1}}"{description}, from=3-1, to=3-3]
	\arrow["{i_{1,2}}"{description}, from=3-3, to=3-5]
	\arrow["{i_{0,j}}"{description, pos=0.4}, from=3-1, to=2-5]
	\arrow["{i_{1,j}}"{description, pos=0.4}, from=3-3, to=2-5]
	\arrow["{\rho _0}"{description}, curve={height=-12pt}, from=3-1, to=1-5]
	\arrow["{\rho_1}"{description}, curve={height=-6pt}, from=3-3, to=1-5]
	\arrow["{\rho _j}"', dashed, from=2-5, to=1-5]
\end{tikzcd}\]

Let $\lambda $ be a cardinal, such that domains of $I$ are $\lambda $-small. The composition $\mathcal{X} \xrightarrow{i_{0,\lambda}} \mathcal{Z}_{\lambda} \xrightarrow{\rho _{\lambda}} \mathcal{Y}$ is isomorphic to $f$ and $i_{0,\lambda} \in I$-cell by the previous propositions.
 
It remains to prove that $\rho _\lambda \in I$-inj. Take a square
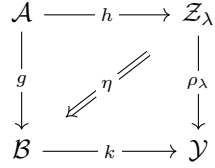
\[\begin{tikzcd}
	{\mathcal{A}} && {\mathcal{Z}_\lambda} \\
	\\
	{\mathcal{B}} && {\mathcal{Y}}
	\arrow["h"{description}, from=1-1, to=1-3]
	\arrow["{\rho _\lambda}"{description}, from=1-3, to=3-3]
	\arrow["g"{description}, from=1-1, to=3-1]
	\arrow["k"{description}, from=3-1, to=3-3]
	\arrow["\eta"{description}, shorten <=11pt, shorten >=11pt, Rightarrow, from=1-3, to=3-1]
\end{tikzcd}\]
As $\mathcal{A}$ is $\lambda $-small, $h$ factors through some stage $\mathcal{Z}_j$ (up to isomorphism). This means, that the back face of the left cube in 

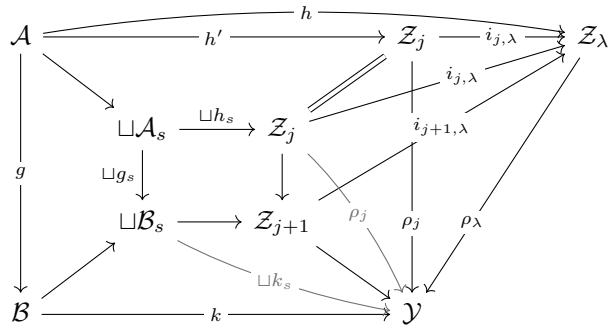
\[\begin{tikzcd}
	{\mathcal{A}} &&& {\mathcal{Z}_{j}} && {\mathcal{Z}_{\lambda}} \\
	& {\sqcup \mathcal{A}_s} & {\mathcal{Z}_{j}} \\
	& {\sqcup \mathcal{B}_s} & {\mathcal{Z}_{j+1}} \\
	{\mathcal{B}} &&& {\mathcal{Y}}
	\arrow["{h'}"{description}, from=1-1, to=1-4]
	\arrow["{\rho _j}"{description, pos=0.7}, from=1-4, to=4-4]
	\arrow["g"{description}, from=1-1, to=4-1]
	\arrow["k"{description}, from=4-1, to=4-4]
	\arrow["{\sqcup g_s}"', from=2-2, to=3-2]
	\arrow[from=3-2, to=3-3]
	\arrow[from=2-3, to=3-3]
	\arrow["{\sqcup h_s}", from=2-2, to=2-3]
	\arrow[from=1-1, to=2-2]
	\arrow[from=4-1, to=3-2]
	\arrow[from=3-3, to=4-4]
	\arrow[Rightarrow, no head, from=2-3, to=1-4]
	\arrow["{i_{j,\lambda }}"{description}, from=1-4, to=1-6]
	\arrow["{i_{j,\lambda }}"{description, pos=0.6}, from=2-3, to=1-6]
	\arrow["{i_{j+1,\lambda }}"{description}, from=3-3, to=1-6]
	\arrow["{\rho _{\lambda}}"{description, pos=0.7}, from=1-6, to=4-4]
	\arrow["{\rho _j}"{description}, color={rgb,255:red,110;green,110;blue,110}, curve={height=-6pt}, from=2-3, to=4-4]
	\arrow["{\sqcup k_s}"{description}, color={rgb,255:red,110;green,110;blue,110}, curve={height=6pt}, from=3-2, to=4-4]
	\arrow["h"{description}, curve={height=-12pt}, from=1-1, to=1-6]
\end{tikzcd}\]
was considered in the formation of $\mathcal{Z}_{j+1}$.
This face is just the gluing of
\[\begin{tikzcd}
	{\mathcal{A}} &&& {\mathcal{Z}_{j}} \\
	& {\mathcal{Z}_{\lambda}} \\
	{\mathcal{B}} &&& {\mathcal{Y}}
	\arrow["{h'}"{description}, from=1-1, to=1-4]
	\arrow["g"{description}, from=1-1, to=3-1]
	\arrow["k"{description}, from=3-1, to=3-4]
	\arrow[""{name=0, anchor=center, inner sep=0}, "{\rho _j}"{description}, from=1-4, to=3-4]
	\arrow[""{name=1, anchor=center, inner sep=0}, "h"{description}, from=1-1, to=2-2]
	\arrow["{i_{j,\lambda}}"{description}, from=1-4, to=2-2]
	\arrow["{\rho _\lambda}"{description}, from=2-2, to=3-4]
	\arrow[shorten <=4pt, shorten >=8pt, Rightarrow, from=2-2, to=3-1]
	\arrow[shorten <=17pt, shorten >=22pt, Rightarrow, from=0, to=2-2]
	\arrow[shorten <=36pt, shorten >=22pt, Rightarrow, from=1-4, to=1]
\end{tikzcd}\]

By (*) the left cube is a commutative (identical) 3-cell, and so is the cone over the $\mathcal{Z}_n$-s. Hence the lift $\mathcal{B}\to \sqcup \mathcal{B}_s \to \mathcal{Z}_{j+1} \to \mathcal{Z}_\lambda $ is a splitting of $\eta $.
\end{proof}

In \cite{filtered} there is an explicit description for filtered 2-colimits in $\mathbf{Cat}$. As a special case we get the following description for the homotopy colimit of the sequence $\mathcal{C}_0\xrightarrow{F_{0,1}} \mathcal{C}_1 \xrightarrow{F_{1,2}} \dots$. Its class of objects is the disjoint union of that of the $\mathcal{C}_i$'s, and an arrow from $(x,i)$ to $(y,j)$ (with $x\in \mathcal{C}_i$ and $y\in \mathcal{C}_j$) is the equivalence class of an arrow $F_{i,k}(x)\xrightarrow{f} F_{j,k}(y)$, where $f$ and $f':F_{i,k'}(x)\to F_{j,k'}(y)$ are equivalent if (assuming $k<k'$) we have $F_{k,k'}(f)=f'$. The induced map in

\[\begin{tikzcd}
	&&& {\mathcal{D}} & {G_k(F_{i,k}(x))} & {G_k(F_{j,k}(y))} \\
	&&&& {G_i(x)} & {G_j(y)} \\
	&&& {\mathcal{C}} & {(x,i)} & {(y,j)} \\
	{\mathcal{C}_0} && {\mathcal{C}_1} && \dots
	\arrow["{F_{0,1}}", from=4-1, to=4-3]
	\arrow["{F_{1,2}}", from=4-3, to=4-5]
	\arrow[from=4-1, to=3-4]
	\arrow[from=4-3, to=3-4]
	\arrow[""{name=0, anchor=center, inner sep=0}, "{G_0}", curve={height=-12pt}, from=4-1, to=1-4]
	\arrow[""{name=1, anchor=center, inner sep=0}, "{G_1}"{pos=0.4}, curve={height=-6pt}, from=4-3, to=1-4]
	\arrow[dashed, from=3-4, to=1-4]
	\arrow[""{name=2, anchor=center, inner sep=0}, "{[f]}"', from=3-5, to=3-6]
	\arrow["{(\eta _{i,k})_x}", from=2-5, to=1-5]
	\arrow["f", from=1-5, to=1-6]
	\arrow["{(\eta _{j,k}^{-1})_y}", from=1-6, to=2-6]
	\arrow[""{name=3, anchor=center, inner sep=0}, dashed, from=2-5, to=2-6]
	\arrow["{\eta_{0,1}}", shift left=5, shorten <=14pt, shorten >=3pt, Rightarrow, from=0, to=1]
	\arrow[shorten <=4pt, shorten >=4pt, maps to, from=2, to=3]
\end{tikzcd}\]
makes the diagram strictly commute when $G_i$'s form a strict cocone, hence we got that this $\mathcal{C}$ is isomorphic to the 1-categorical colimit described in section 3. As it was proved to be coherent, we have that transfinite compositions (of strict sequences) in the 2-categorical sense can be chosen to be 1-categorical colimits. Therefore in the inductive proof of Theorem \ref{smallob} the sequence $\mathcal{Z}_0 \xrightarrow{i_{0,1}} \dots $ can be chosen to be strict, so any $\lambda $ with $cf(\lambda )> sup \{|dom(f)| :f\in I\}$ works. Finally we proved:

\begin{theorem}
Let $I$ be a small set of coherent functors. Given a coherent functor $\mathcal{C}\xrightarrow{m} \mathcal{E}$ it is isomorphic to a composition $\mathcal{C}\xrightarrow{f}\mathcal{D}\xrightarrow{g}\mathcal{E}$ where $f\in I$-cell and $g\in I$-inj. In particular $f\in I$-cof.
\end{theorem}

\printbibliography

@misc{filtered,
  author =       "Delphine Dupont",
  title =        "Interchange of filtered 2-colimits and finite
2-limits",
  year =         "2009",
  archivePrefix = {arXiv},
  eprint = {0904.1553}
}

@book{makkai,
    author    = "Michel Makkai and Gonzalo E. Reyes",
    title     = "First Order Categorical Logic",
    year      = "1977",
    publisher = "Springer",
    address   = "Berlin, Heidelberg",
    DOI       = "https://doi.org/10.1007/BFb0066201"
}

@book{rosicky,
    author    = "Jiří Adámek and Jiří Rosický",
    title     = "Locally Presentable and Accessible Categories",
    year      = "1994",
    publisher = "Cambridge University Press",
    address   = "London",
    series   = "London Mathematical Society Lecture Note Series",
    volume    = "189"
}

@book{sheaves,
    author    = "Saunders MacLane and Ieke Moerdijk",
    title     = "Sheaves in Geometry and Logic",
    year      = "1992",
    publisher = "Springer",
    address   = "New York, NY",
    DOI       = "https://doi.org/10.1007/978-1-4612-0927-0"
}

@book{hovey,
    author    = "Mark Hovey",
    title     = "Model categories",
    year      = "1991",
    adress = "Wesleyan University, Middletown",
    URL      = "https://people.math.rochester.edu/faculty/doug/otherpapers/hovey-model-cats.pdf"
}

@book{lurietopos,
    author    = "Jacob Lurie",
    title     = "Higher Topos Theory",
    year      = "2009",
    publisher = "Princeton University Press",
    series   = "Annals of Mathematics Studies",
    volume    = "170"
}

@book{maclane,
    author    = "Saunders Mac Lane",
    title     = "Categories for the Working Mathematician",
    year      = "1971",
    publisher = "Springer",
    address   = "New York, NY",
    DOI       = "https://doi.org/10.1007/978-1-4612-9839-7"
    }

@misc{joyal,
  author        = {André Joyal},
  title         = {The Theory of Quasi-Categories and its Applications},
  year          = {2008},
  URL       = "https://mat.uab.cat/~kock/crm/hocat/advanced-course/Quadern45-2.pdf"
}

\end{document}